\documentclass[11pt]{article}

\usepackage{amsmath, amsthm}
\usepackage{amsmath, amsfonts}
\usepackage{amsmath, amssymb}
\usepackage{amsmath}
\usepackage[all]{xy}

\newtheorem{theorem}{Theorem}[subsection]
\newtheorem{definition}[theorem]{Definition}
\newtheorem{definition-lemma}[theorem]{Definition/Lemma}
\newtheorem{definition-explanation}[theorem]{Definition/Explanation}
\newtheorem{explanation-definition}[theorem]{Explanation/Definition}
\newtheorem{definition-fact}[theorem]{Definition/Fact}
\newtheorem{definition-notation}[theorem]{Definition/Notation}
\newtheorem{lemma}[theorem]{Lemma}
\newtheorem{lemma-definition}[theorem]{Lemma/Definition}
\newtheorem{proposition}[theorem]{Proposition}
\newtheorem{corollary}[theorem]{Corollary}
\newtheorem{remark}[theorem]{\it Remark}
\newtheorem{remark-notation}[theorem]{\it Remark/Notation}
\newtheorem{conjecture}[theorem]{Conjecture}

\newtheorem{example}[theorem]{Example}
\newtheorem{example-definition}[theorem]{Example/Definition}
\newtheorem{notation}[theorem]{Notation}

\newtheorem{definition-prototype}[theorem]{Definition-Prototype}

\numberwithin{equation}{subsection}

\newtheorem{stheorem}{Theorem}[section]
\newtheorem{sdefinition}[stheorem]{Definition}
\newtheorem{sdefinition-lemma}[stheorem]{Definition/Lemma}
\newtheorem{sdefinition-explanation}[stheorem]{Definition/Explanation}
\newtheorem{sexplanation-definition}[stheorem]{Explanation/Definition}
\newtheorem{sdefinition-fact}[stheorem]{Definition/Fact}
\newtheorem{sdefinition-notation}[theorem]{Definition/Notation}
\newtheorem{slemma}[stheorem]{Lemma}
\newtheorem{slemma-definition}[stheorem]{Lemma/Definition}
\newtheorem{sproposition}[stheorem]{Proposition}

\newtheorem{sremark}[stheorem]{\it Remark}
\newtheorem{sremark-notation}[stheorem]{\it Remark/Notation}

\newtheorem{sexample-definition}[stheorem]{Example/Definition}

\newtheorem{sdefinition-prototype}[stheorem]{Definition-Prototype}

\newtheorem{ssdefinition-lemma}[sstheorem]{Definition/Lemma}
\newtheorem{ssdefinition-explanation}[sstheorem]{Definition/Explanation}
\newtheorem{ssexplanation-definition}[sstheorem]{Explanation/Definition}
\newtheorem{ssdefinition-fact}[stheorem]{Definition/Fact}
\newtheorem{ssdefinition-notation}[theorem]{Definition/Notation}

\newtheorem{sslemma-definition}[sstheorem]{Lemma/Definition}

\newtheorem{ssremark-notation}[sstheorem]{\it Remark/Notation}

\newtheorem{ssexample-definition}[sstheorem]{Example/Definition}

\newtheorem{ssdefinition-prototype}[sstheorem]{Definition-Prototype}

\newcommand{\Br}{\mbox{\it Br}\,}

\newcommand{\Coh}{\mbox{\it Coh}\,}

\newcommand{\DCG}{\mbox{\it DCG}\,}

\newcommand{\End}{\mbox{\it End}\,}
 
\newcommand{\Endsheaf}{\mbox{\it ${\cal E}\!$nd}\,}

\newcommand{\Extsheaf}{\mbox{\it ${\cal E}$xt}\,}

\newcommand{\FM}{\mbox{${\cal F}$\!${\cal M}$}\,}
 \newcommand{\smallFM}{\mbox{\small${\cal F}$\!${\cal M}$}\,}

\newcommand{\HN}{\mbox{\it HN}\,}
\newcommand{\Hom}{\mbox{\it Hom}\,}

\newcommand{\Id}{\mbox{\it Id}\,}

\newcommand{\Imaginary}{\mbox{\it Im}\,}
\newcommand{\Image}{\mbox{\it Im}\,}

\newcommand{\Int}{\mbox{\it Int}\,}

\newcommand{\Ker}{\mbox{\it Ker}\,}
\newcommand{\Kthree}{\mbox{\it K3}}
\newcommand{\KCone}{\mbox{\it KCone}\,}

\newcommand{\NE}{\mbox{\it NE}\,}

\newcommand{\Poly}{\mbox{\it Poly}\,}
  \newcommand{\scriptsizePoly}{\mbox{\scriptsize\it Poly}\,}

\newcommand{\Quot}{\mbox{\it Quot}\,}

\newcommand{\Real}{\mbox{\it Re}\,}

\newcommand{\Spec}{\mbox{\it Spec}\,}
 \newcommand{\boldSpec}{\mbox{\it\bf Spec}\,}

\newcommand{\Stab}{\mbox{\it Stab}\,}

\newcommand{\Supp}{\mbox{\it Supp}\,}

\newcommand{\Tor}{\mbox{\it Tor}}

\newcommand{\smallZss}{\mbox{\small\it $Z$-$ss$}\,}
 \newcommand{\scriptsizeZss}{\mbox{\scriptsize\it $Z$-$ss$}\,}
 \newcommand{\tinyZss}{\mbox{\tiny\it $Z$-$ss$}\,}

\newcommand{\ch}{\mbox{\it ch}\,}

\newcommand{\degree}{\mbox{\it deg}\,}

\newcommand{\dimm}{\mbox{\it dim}\,}

\newcommand{\gr}{\mbox{\it gr}\,}

\newcommand{\scriptsizenaive}{\mbox{\scriptsize\it naive}}

\newcommand{\pr}{\mbox{\it pr}}

\newcommand{\td}{\mbox{\it td}\,}

\newcommand{\longleftaarrow}{\longleftarrow\hspace{-3ex}\longleftarrow}
\newcommand{\longrightaarrow}{\longrightarrow\hspace{-3ex}\longrightarrow}

\begin{document}

\enlargethispage{24cm}

\begin{titlepage}

$ $

\vspace{-1.5cm} 

\noindent\hspace{-1cm}
\parbox{6cm}{\small  December 2012}\
   \hspace{6cm}\
   \parbox[t]{6cm}{yymm.nnnn [math.AG] \\
                D(10.1): D1, central charge,\\
                $\mbox{\hspace{3.8em}} $  auxiliary moduli stack}

\vspace{2cm}

\centerline{\large\bf
  A mathematical theory of D-string world-sheet instantons,}
\vspace{1ex}
\centerline{\large\bf
 I: Compactness of the stack of $Z$-semistable  Fourier-Mukai transforms}
\vspace{1ex}
\centerline{\large\bf
  from a compact family of nodal curves to a projective Calabi-Yau $3$-fold}


\vspace{3em}

\centerline{\large
  Chien-Hao Liu
   \hspace{1ex} and \hspace{1ex}
  Shing-Tung Yau
}

\vspace{4em}

\begin{quotation}
\centerline{\bf Abstract}

\vspace{0.3cm}

\baselineskip 12pt  
{\small
  In a suitable regime of superstring theory, D-branes in a Calabi-Yau space
   and their most fundamental behaviors can be nicely described
   mathematically through morphisms from Azumaya spaces
   with a fundamental module to that Calabi-Yau space.
  In the earlier work [L-L-S-Y]
    (D(2): arXiv:0809.2121 [math.AG], with Si Li and Ruifang Song)
	 from the project,
   we explored this notion for the case of D1-branes (i.e. D-strings)  and
    laid down some basic ingredients toward understanding
	  the notion of D-string world-sheet instantons in this context.
  In this continuation, D(10), of D(2), we move on to construct
   a moduli stack of semistable morphisms from Azumaya nodal curves
   with a fundamental module to a projective Calabi-Yau $3$-fold $Y$.
  In this Part I of the note, D(10.1),
   we define
              the notion of twisted central charge $Z$
			     for Fourier-Mukai transforms of dimension $1$ and width $[0]$
				  from nodal curves
            and the associated stability condition on such transforms
   and prove that for a given compact stack of nodal curves $C_{\cal M}/{\cal M}$,
     the stack $\smallFM^{1,[0];\tinyZss}_{C_{\cal M}/{\cal M}}(Y,c)$
	 of $Z$-semistable Fourier-Mukai transforms
	 of dimension $1$ and width $[0]$
	  from nodal curves in the family $C_{\cal M }/{\cal M}$
	  to $Y$ of fixed twisted central charge $c$ is compact.
  For the application in the sequel D(10.2),
   $C_{\cal M}/{\cal M}$ will contain
    $C_{\overline{\cal M }_g}/\overline{\cal M}_g$ as a substack  and
	$\smallFM^{1,[0];\tinyZss}_{C_{\cal M}/{\cal M}}(Y,c)$
       in this case will play a key role in defining stability conditions
	   for morphisms from arbitrary Azumaya nodal curves
	    (with the underlying nodal curves not necessary
     		in the family $C_{\cal M}/{\cal M}$)
       to $Y$.
 } 
\end{quotation}

\vspace{2em}

\baselineskip 12pt
{\footnotesize
\noindent
{\bf Key words:} \parbox[t]{14cm}{D-string, world-sheet instanton;
      Azumaya nodal curve, fundamental module, morphism; \\
	  central charge, stability condition, wall,  chamber structure;
	  compactness of moduli.
 }} 

\bigskip

\noindent {\small MSC number 2010: 14D20, 81T30, 14F05, 14N35, 14A22.
} 

\bigskip

\baselineskip 10pt
{\scriptsize
\noindent{\bf Acknowledgements.}
We thank Andrew Strominger and Cumrun Vafa for lectures and discussions
 that influence our understanding of strings, branes, and gravity.
C.-H.L.\ thanks in addition
 Si Li , Ruifang Song
    for discussions in the biweekly Saturday D-brane Working Seminar, spring 2008;
 Alina Marian, Yu-jong Tzeng, Baosen~Wu
    for discussions and topic courses in later years (fall 2009, fall 2010, fall 2011)
    that influenced his thoughts on enumerative algebraic geometry and coherent sheaves;
 Paul Aspinwall
    for a clarification of central charge in his notes (December 2011);
 Arend Bayer for a discussion on stability (May 2012);
 Chieh-Cheng Jason Lo
   for the communication of his thesis  and an expert's concise guide
   on the notion of stability since the work of Tom Bridgeland in 2002
   (September 2012)    and
 Cumrun~Vafa for a clarification of a subtle point
   on stability of D-branes from string theory aspect  (December 2012)
   and reference guide along the way.
He thanks also
 Hao~Xu, Adam Jacob for discussions and topic courses
   that drew his attention to the notion of stability in other contents
  (spring 2012, fall 2012);
 Mboyo Esole
   for the discussions and literature guide in and after his topic course in F-theory
   (fall 2012);
 Siu-Cheong Lau, S.L., Yi Li,  Yu-Shen Lin, Li-Sheng Tseng, Ke Zhu
   for discussions on other issues related to D-branes and Gromov-Witten theory;
 Erel~Levine, Andrew~Strominger/St\'{e}phene Detournay,
 Noam Elkies, C.V.\
   for other basic/topic courses in spring and fall 2012;
 Master Guo-Guang~Shr,  Master Hui-Min Shr
   for hospitality and discussions on the notion of space-times in large dimensions;
 Hsin-Yu~Chen, Rui-Hsieh~Shen, Chao-Tze Liu,  Ann~Willman,
 Rev.~Campbell~Willman and Betty
   for hospitality while parts of D(10) are under brewing;  and
 Ling-Miao Chou for moral support.
The project is supported by NSF grants DMS-9803347 and DMS-0074329.
} 

\end{titlepage}

\newpage

\begin{titlepage}

$ $

\vspace{12em} 

\centerline{\small\it
 Chien-Hao Liu dedicates this subseries D(10), to be completed, to}
\centerline{\small\it
 his uncle Prof.~Pin-Hsiung Liu$^{\dagger\ast}$, aunt Ms.~Rui-Be Lin,}
\centerline{\small\it
 and cousin Master Guo-Guang Shr (Te-Ru Liu)}
\centerline{\small\it
 for their hospitality in his first year of college,}
\centerline{\small\it
 years of communications with Te-Ru,}
\centerline{\small\it
 and lots of cherished memories.}

\vspace{12em}

\baselineskip 11pt

{\footnotesize
\noindent
$^{\dagger}$Deceased, winter  2004.  (1925 -- 2004)

\medskip

\noindent $^{\ast}$(From C.H.L.)
{From} the preschool years when I was beginning to be aware of this world,
the big family of my father's generation seemed to already fall apart,
like the story in the Chinese classic
 ``Hong Lou Meng (Dream of the Red Chamber)",
 composed by Hsue-Chin Tsao in the middle of the 18th century.
My third uncle somehow stood out as the most educated
 and also the most independent among my father's eleven siblings;
 and left the big shadow of my grandfather  --  being a medical doctor and a local elite --
 to pursue his own study and dream on anthropology.
Later he got a Harvard Yenching Institute scholarship to visit here
 as a scholar during 1964 -1965.
There, through the interaction with a linguist John Harvey,
 he got deep insight toward a resolution of a kinship problem
 that had puzzled anthropologists for long,
 using {\it combinatorial group theory} (!!!); [L1].
Starting from there, {\it mathematic anthropology}
  became his marking territory,
  a field that would likely deter most anthropologists; [L2].
I remember in my first year of high school, he came to visit us
 and taught my brother and me to play `go'.
But my real contact with him won't  begin
 until I got to college and moved in to his house
  and lived together with my aunt and cousin for a year.
I observed a scholar that is driven not by fame but by his dream.
While writing this,
 lots of cherished memories with my aunt and, particularly, my cousin
 also re-emerge.
I thank them all for shredding away the clumsiness of a country kid
 in a big city.
In preparing this note, I read through parts of [L-G-H] again.
In the last sentence of the last article of the collection,
it was pondered upon what legacy my uncle had left behind.
This is a question that likely can have only personal answers.
But to me, in view of being in the middle of an extremely demanding project
    that still has a long way to go,
 the answer couldn't be more obvious:
 He has given me a role model of perseverance and persistence to pursue one's dream.
 I thus dedicate this D(10), yet to be completed, to the memory of my uncle,
  and to my aunt and cousin Te-Ru/Master Guo-Guang.

\bigskip

\parbox[t]{4em}{[L1]}\
  \parbox[t]{40em}{P.-H.~Liu, $\;\;$
    {\sl Murngin: A mathematical solution}, Mono.\ Ser.~B, no.~2,
	  Institute of Ethnology, Academic Sinica, Taipei, Taiwan, 1969.}

\medskip
	
\parbox[t]{4em}{[L2]}\
  \parbox[t]{40em}{--------, $\;\;$
    {\sl Foundations of kinship mathematics}, Mono.\ Ser.~A, no.~28,
	 Institute of Ethnology, Academic Sinica, Taipei, Taiwan, 1986.} 	
	
\medskip

\parbox[t]{4em}{[L-G-H]}\
  \parbox[t]{40em}{  M.-r.~Lin, P.-y.~Guo and C.-H.~Huang eds., $\;\;$
   {\sl Essays in honor of Professor Pin-Hsiung Liu},
      Institute of Ethnology, Academic Sinica, Taipei, Taiwan, 2008.}
} 

\end{titlepage}


\newpage
$ $

\vspace{-3em}

\centerline{\sc
 D-string world-sheet instantons I: Auxiliary Moduli Stack
 } %

\vspace{2em}


\begin{flushleft}
{\Large\bf 0. Introduction and outline.}
\end{flushleft}
 In a suitable regime of superstring theory, D-branes in a Calabi-Yau space
   and their most fundamental behaviors can be nicely described
   mathematically through morphisms from Azumaya spaces
   with a fundamental module to that Calabi-Yau space.
  In the earlier work [L-L-S-Y]
    (D(2): arXiv:0809.2121 [math.AG], with Si Li and Ruifang Song)
	 from the project,
   we expored this notion for the case of D1-branes (i.e. D-strings)  and
    laid down some basic ingredients toward understanding
	  the notion of D-string world-sheet instantons in this context.
  In this continuation, D(10), of D(2), we move on to construct
   a moduli stack of semistable morphisms from Azumaya nodal curves
   with a fundamental module to a projective Calabi-Yau $3$-fold $Y$.
  In this Part I of the note, D(10.1),
   we define
              the notion of twisted central charge $Z$
			     for Fourier-Mukai transforms of dimension $1$ and width $[0]$
				  from nodal curves
            and the associated stability condition on such transforms
   and prove that for a given compact stack of nodal curves $C_{\cal M}/{\cal M}$,
     the stack $\smallFM^{1,[0];\tinyZss}_{C_{\cal M}/{\cal M}}(Y,c)$
	 of $Z$-semistable Fourier-Mukai transforms
	 of dimension $1$ and width $[0]$
	  from nodal curves in the family $C_{\cal M }/{\cal M}$
	  to $Y$ of fixed twisted central charge $c$ is compact.
  For the application in the sequel D(10.2),
   $C_{\cal M}/{\cal M}$ will contain
    $C_{\overline{\cal M }_g}/\overline{\cal M}_g$ as a substack  and
	$\smallFM^{1,[0];\tinyZss}_{C_{\cal M}/{\cal M}}(Y,c)$
       in this case will play a key role in defining stability conditions
	   for morphisms from arbitrary Azumaya nodal curves
	    (with the underlying nodal curves not necessary
     		in the family $C_{\cal M}/{\cal M}$)
       to $Y$.

 \bigskip
	
\noindent{\it Remark 0.1.}{\it $[$general Fourier-Mukai transform$]$.}	
 We should remark that
 it is very natural to anticipate
  a generalized setting and compactness result of the current note D(10.1),
  in which
   general Fourier-Mukai transforms (cf.\ [Hu]) associated to
      objects in the bounded derived category $D^b(\Coh(C\times Y))$
      of coherent sheaves are considered   and
   the general stability conditions from the works [Bay], [Br], [G-K-R], [Ru],
   a generalization of the Kleiman's Boundedness Criterion ([Kl])
     and of the related inequalities that bounds $h^0$ by slopes ([LP], [Ma2], [Si]),
    and the result of valuative criterion from
	Dan Abramovich and Alexander Polishchuk [A-P]  and Jason Lo [Lo1], [Lo2]
	 are taken as the foundation.

\bigskip

\bigskip

\noindent
{\bf Convention.}
 Standard notations, terminology, operations, facts in
  (1) string theory/D-branes, supersymmetry;$\,$
  (2) stacks;$\,$
  (3) Fourier-Mukai transforms;$\,$
  (4) moduli spaces of sheaves$\,$
 can be found respectively in$\,$
  (1) [Po1], [Po2], [Bac], [Jo]; [Ar], [W-B];$\,$
  (2) [L-MB];$\,$
  (3) [Hu]
  (4) [H-L].
 \begin{itemize}
  \item[$\cdot$]
   All varieties, schemes and their products are over ${\Bbb C}$;
   a `{\it curve}' means a $1$-dimensional proper scheme over ${\Bbb C}$.

  \item[$\cdot$]
   The `{\it support}' $\Supp({\cal F})$
    of a coherent sheaf ${\cal F}$ on a scheme $Y$
    means the {\it scheme-theoretical support} of ${\cal F}$
   unless otherwise noted;
   ${\cal I}_Z$ denotes the {\it ideal sheaf} of
    a subscheme of $Z$ of a scheme $Y$.

  \item[$\cdot$]
   The `{\it polarization class group}' of a curve $C$
   means the same as
   the `{\it degree class group}' $\DCG(C)$ of $C$.
   The former emphasizes that its elements are represented
     by ample line bundles
   while the latter emphasizes its discrete nature.

  \item[$\cdot$]
   For the projection maps
     $\pr_1:C\times Y\rightarrow C$ and $\pr_2:C\times Y\rightarrow Y$,
   $\pr_1^{\ast}L$ (resp.\ $\pr_2^{\ast}(B+\sqrt{-1}J)$)
    may be denoted still by $L$ (resp.\ $B+\sqrt{-1}J$)	
    for the simplicity of expressions when there is no chance of confusion.
  Here, $L$	is a line bundle or a line bundle class on $C$ and
    $B+\sqrt{-1}J$ is a complexified K\"{a}hler class on $Y$.
	
  \item[$\cdot$]
   A curve class in $A_1(C\times Y)$ and its image in $N_1(C\times Y)_{\Bbb Z}$
   are denoted the same when we need the latter.
  
  \item[$\cdot$]
  {\it Central charge functional} $Z$
    vs.\ {\it variety/scheme} or {\it cycle/cycle class} $Z$.

  \item[$\cdot$]
  {\it Line bundle} or {\it line-bundle class} $L$
   vs.\  {\it calibrated cycle} $L$,
            (especially, for special {\it L}agrangian).

  \item[$\cdot$]
   The word `{\it twist}' in this note means an alteration by a line-bundle,
   except in Remark~1.3, where it means an effect due to $B$-field.
   (The two concepts are completely different.)

  \item[$\cdot$]
   The current note continues the study in
    [L-L-S-Y] (arXiv:0809.2121 [math.AG], D(2)).
   A partial review of D-branes and Azumaya noncommutative geometry
    is given in [L-Y3] (arXiv:1003.1178 [math.SG], D(6)) and
    [Liu$_{CH}$] (arXiv:1112.4317 [math.AG]).
   Notations and conventions follow these early works when applicable.
 \end{itemize}

\bigskip

\bigskip

\begin{flushleft}
{\bf Outline.}
\end{flushleft}
\nopagebreak
{\small
\baselineskip 12pt  
\begin{itemize}
 \item[0.]
  Introduction.

 \item[1.]
  D-string world-sheet instantons,
  morphisms from Azumaya nodal curves with a fundamental module, and
  Fourier-Mukai transforms from nodal curves.
  \begin{itemize}
   \item[{\boldmath $\cdot$}]
    D-branes and morphisms from Azumaya spaces with a fundamental module.

   \item[{\boldmath $\cdot$}]
    Morphisms from Azumaya schemes with a fundamental module\\
         versus Fourier-Mukai transforms.
		
   \item[{\boldmath $\cdot$}]
    Stable morphisms from Azumaya nodal curves with a fundamental module\\
         as D-string world-sheet instantons.
  \end{itemize}

 \item[2.]
  Twisted central charges and stability conditions on Fourier-Mukai transforms
  from nodal curves.
  \begin{itemize}
   \item[2.1]
     Twisted central charges of D-string world-sheet instantons.
	 \begin{itemize}
	  \item[{\boldmath $\cdot$}]
       Central charges of D-branes in a Calabi-Yau $3$-fold and why we need a twist.
	
	  \item[{\boldmath $\cdot$}]
       Twisted central charges of D-string world-sheet instantons.
	
	  \item[{\boldmath $\cdot$}]
       Basic properties.
     \end{itemize}
	
   \item[2.2]
    Stability conditions on Fourier-Mukai transforms
     of dimension $1$ and width $[0]$ from nodal curves.
	 \begin{itemize}
	  \item[{\boldmath $\cdot$}]
	   Stability conditions defined by a central charge functional $Z$.
	
	  \item[{\boldmath $\cdot$}]
	   The Harder-Narasimhan filtration with respect to $Z$.

	  \item[{\boldmath $\cdot$}]
	   Jordan-H\"{o}lder filtrations and $S$-equivalence with respect to $Z$.
	 \end{itemize}
	
   \item[2.3]
    A chamber structure on the space of  stability conditions.
	\begin{itemize}
	 \item[{\boldmath $\cdot$}]
	  The space of stability conditions $\Stab^{1,[0]}(C\times Y)$.
	
	 \item[{\boldmath $\cdot$}]
	  Walls and chambers on $\Stab^{1,[0]}(C\times Y)$.
	
	 \item[{\boldmath $\cdot$}]
	   Local finiteness of walls.
	
	 \item[{\boldmath $\cdot$}]
	   Behavior of the moduli stack of stable objects when crossing an actual wall:
       a conjecture.
	\end{itemize}
  \end{itemize}

 \item[3.]
  Compactness of the moduli stack
	$\smallFM^{1,[0];\tinyZss}_{C_{\cal M}/{\cal M}}(Y,c)$
   of $Z$-semistable Fourier-Mukai transforms.
   \begin{itemize}
     \item[{\boldmath $\cdot$}]
	  Boundedness of
      $\FM^{1,[0];\scriptsizeZss}_{C_{\cal M}/{\cal M}}(Y;c)$.
	
     \item[{\boldmath $\cdot$}]
	  Completeness of $\FM^{1,[0];\scriptsizeZss}_{C_{\cal M}/{\cal M}}(Y;c)$.
   \end{itemize}
 %
 %
 %
 %
 %
\end{itemize}
} 

\newpage

\section{D-branes, morphisms from Azumaya spaces with a fundamental module,
				  Fourier-Mukai transforms, and D-string world-sheet instantons}

In this section,
we review tersely how we come to this to set up basic terminologies and notations
 and bring out the notion of `D-string world-sheet instantons'.	
Readers are referred to
  [L-Y1] (D(1)),  [L-L-S-Y] (D(2))
    for more thorough related discussions,
  [L-Y3] (D(6)) and [Liu$_{CH}$] for a review of the project up to D(9.1),  and
  [Hu] for Fourier-Mukai transforms.

\bigskip

\begin{flushleft}
{\bf D-branes and morphisms from Azumaya spaces with a fundamental module.}
\end{flushleft}
The open-string-induced matrix-valued-type enhancement of the scalar fields
  on a D-brane world-volume $X$  for coincident D-branes in a space-time $Y$
 that describe deformations of the brane world-volume in $Y$
 motivated string-theorists,
   Pei-Ming Ho and Yong-Shi Wu [H-W] in particular,
 to propose a fundamental matrix-type noncommutative geometry
  on the D-brane world-volume
  (i.e.\ D-branes as `quantum space' in the language of [H-W]);
see also related discussion of Michael Douglas [Dou1].
This appearance of matrix-type noncommutative structure
   on the D-brane world-volume $X$
   ({\it rather than} directly on the target space-time $Y$)
  turns out to be also a most natural interpretation of what's going on
  from Grothendieck's viewpoint of algebraic geometry and
  his theory of schemes and morphisms between them.
A D-brane in this content is then described by a morphism from
 a scheme with a matrix-type noncommutative structure sheaf
 (i.e. an Azumaya scheme $(X,{\cal O}_X^{A\!z})$)
 together with a fundamental ${\cal O}_X^{A\!z}$-module ${\cal E}$
 to $(Y, {\cal O}_Y)$, where ${\cal O}_Y$ is the structure sheaf of $Y$
  in either commutative or noncommutative setting; in notation/symbol,
 $$
   \varphi\; :\;  (X,{\cal O}_X^{A\!z}; {\cal E})\;
     \longrightarrow\;   (Y,{\cal O}_Y)\,,
 $$
 with a built-in isomorphism
   ${\cal O}_X^{A\!z}\simeq \Endsheaf_{X}({\cal E})$.
In true contents, this means
 a contravariant gluing system of ring-homomorphisms
 $$
    {\cal O}_X^{A\!z}\;
	   \longleftarrow\; {\cal O}_Y\; :\; \varphi^{\sharp}\,,
$$
which in general {\it does not} induce any morphisms directly
   from $X$ to $Y$.
It is through $\varphi^{\sharp}$ that
 the ${\cal O}_X^{A\!z}$-module ${\cal E}$
 can be pushed forward to an ${\cal O}_Y$-module,
 in notation $\varphi_{\ast}{\cal E}$, on $Y$.
Despite the language difference and different level of developments,
 the same idea applies to
  nonsupersymmetric D-branes,
 (supersymmetric) D-branes of  B-type
     (cf.\ the above setting in algebraic geometry),   and
 (supersymmetric) D-branes of A-type ([L-Y4] (D(7))).

\bigskip

\begin{flushleft}
{\bf Morphisms from Azumaya schemes with a fundamental module\\
         versus Fourier-Mukai transforms.}
\end{flushleft}
When the target space $Y$ is a commutative scheme
  and ${\cal E}_X$ is locally free ${\cal O}_X$-module,
 then 
  associated to a morphism  
  $\varphi :(X,{\cal O}_X^{A\!z};{\cal E})
                                               \rightarrow (Y,{\cal O}_Y)$
  is the following diagram 
  $$
   \xymatrix{
   {\cal O}_X^{A\!z}= \Endsheaf{\cal E}\\
        {\cal A}_{\varphi}\;:=\,
         		\Image\varphi^{\sharp}\ar@{^{(}->}[u]  
                   				&&& {\cal O}_Y\ar[lll]_-{\varphi^{\sharp}}\\
   {\cal O}_X\rule{0ex}{3ex}  \ar@{^{(}->}[u]								&&&&,
   } 
  $$
  which defines a subscheme 
    $X_{\varphi}:= \boldSpec{\cal A}_{\varphi}\subset X\times Y$
	together with a coherent sheaf $\tilde{\cal E}_{\varphi}$ 
	 supported on $X_{\varphi}$, 
   which is simply the ${\cal O}_X^{A\!z}$-module ${\cal E}$ 
    regarded as an ${\cal A}_{\varphi}$-module.
 $\tilde{\cal E}_{\varphi}$ is called the {\it graph} of  the morphism $\varphi$.
It is a coherent sheaf on $X\times Y$
 that is flat over $X$, of relative dimension $0$.
Conversely, given such a coherent sheaf on $X\times Y$,
 a morphism $\varphi:(X,{\cal O}_X^{A\!z};{\cal E})\rightarrow Y$
 can be constructed from $\tilde{\cal E}$ by taking 
 \begin{itemize}
  \item[$\cdot$]
    ${\cal E}=\pr_{1\,\ast}\tilde{\cal E}$,
	
  \item[$\cdot$] 	
   ${\cal O}_X^{A\!z}=\Endsheaf_{{\cal O}_X}({\cal E})$,  and

  \item[$\cdot$]   
   $\varphi^{\sharp}:{\cal O}_Y\rightarrow {\cal O}_X^{A\!z}\,$ 
   is defined by the composition  
   $$
     {\cal O}_Y 
        \xrightarrow{\hspace{1em}pr_2^{\sharp}\hspace{1em}} 
      {\cal O}_{X\times Y}
        \xrightarrow{\hspace{1.2em}\iota^{\sharp}\hspace{1.2em}} 
		{\cal O}_{Supp(\tilde{\cal E})}\;
		\hookrightarrow\; {\cal O}_X^{A\!z}\,.
   $$
 \end{itemize}  
Here, 
  $X\xleftarrow{\pr_1} X\times Y \xrightarrow{\pr_2} $  
      are the projection maps, 
  $\iota:\Supp(\tilde{\cal E})\rightarrow X\times Y$	  
    is the embedding of the subscheme,   and 
note that $\Supp(\tilde{\cal E})$ is affine over $X$.
The subseries D(10) studies issues in the special case $X$ is a nodal curve
 parameterized by a compact Artin stack ${\cal M}$.

Resume the general discussion.
Treating $\tilde{\cal E}$ as an object in the bounded derived category
 $D^b(\Coh(X\times Y))$ of coherent sheaves on $X\times Y$,
 $\tilde{\cal E}$ defines a Fourier-Mukai transform
 $\Phi_{\tilde{\cal F}} : D^b(\Coh(X))\rightarrow D^b(\Coh(Y))$,
 in short name, a {\it Fourier-Mukai transform from $X$ to $Y$}.
In this way, the data that specifies a morphism
 $\phi:(X,{\cal O}_X^{A\!z};{\cal E})\rightarrow Y$
 is matched to a data that specifies a special kind of Fourier-Mukai transform.

\bigskip

\begin{sdefinition}
 {\bf [support, dimension, width of Fourier-Mukai transform].} {\rm
For a general $\tilde{\cal F}^{\bullet}\in D^b(\Coh(X\times Y))$,
 we define
   the {\it (scheme-theoretical) support} $\Supp(\tilde{\cal F})$
       of  $\tilde{\cal F}^{\bullet}$
     to be the (scheme-theoretical) support of
	   $\oplus_iH^i(\tilde{\cal F}^{\bullet})$,
   the {\it dimension} $\dimm\tilde{\cal F}^{\bullet}$
      of $\tilde{\cal F}^{\bullet}$
     to be the dimension $\dimm(\Supp({\cal F}^{\bullet}))$,  and
   the {\it width} of  $\tilde{\cal F}^{\bullet}$ to be the interval
      $[i, j]$ such that
	    $H^i(\tilde{\cal F}^{\bullet})\ne 0$,
	    $H^j(\tilde{\cal F}^{\bullet})\ne 0$,  	and
		$H^k(\tilde{\cal F}^{\bullet})= 0$,  for $k\notin [i,j]$.
 We'll denote the width $[i,i]$ by $[i]$.		
}\end{sdefinition}

\bigskip

\noindent
Thus, for $X$ fixed of pure dimension $d$,
   the stack of morphisms $(X,{\cal O}_X^{A\!z};{\cal E})\rightarrow Y$
   is embedded in the stack of Fourier-Mukai transforms from $X$ to $Y$
   of dimension $d$ and width $[0]$;
 the latter is identical to the stack of $d$-dimensional  coherent sheaves on $X\times Y$.
Similar statement holds for $X$ not fixed.

\bigskip

After the above review, let us turn to the focus of this subseries D(10):
{\it The case of D1-branes (i.e.\ D-strings)}.

\bigskip

\bigskip

\begin{flushleft}
{\bf Stable morphisms from Azumaya nodal curves with a fundamental module\\
         as D-string world-sheet instantons.}
\end{flushleft}
An {\it instanton} in a field theory is by definition a field configuration
  that is localized both in space and in time.
Thus, when an (Euclidean) D$p$ -brane world-volume wraps around a $(p+1)$-cycle
  in a Calabi-Yau $3$-fold, it looks like an instanton in the $4$-dimensional  effective field theory
  from the compactification of a Type II superstring theory on that Calabi-Yau $3$-fold.
Due to its origin, it is called a {\it D-brane world-volume instanton}
  in the $4$-dimensional field theory.
(See e.g.,
    [B-B-S] for various dimensional brane world-volume instantons  and
	[B-C-K-W] for a review of D-instantons in various contents.)
Recall now how the {\it Gromov-Witten theory of stable maps} from nodal curves
    (resp.\ bordered Riemann surface)
      to a Calabi-Yau space $Y$
	(resp.\ a Calabi-Yau space $Y$ with a calibrated or holomorphic cycle $L\subset Y$)
    is related to string-theorists'
	 {\it (fundamental) closed (resp.\ open) string world-sheet instantons},
	   e.g.\ [C-K], [C-dlO-G-P], [D-S-W-W], [K-K-L-MG], [K-L], [Liu$_{CC}$] and [O-V].
Now that we have realized the moving of a D-string in a space-time
 as a morphism from an Azumaya nodal curve with a fundamental module
  (i.e.\ Azumaya Riemann surface with nodes
            in the complex geometry language) to that space-time,
it is very natural to regard
 a {\it stable morphism} from an Azumaya nodal curve with a fundamental module
  to a Calabi-Yau space $Y$ as giving a {\it D-string world-sheet instanton}
  of the $4$-dimensional effective field theory,
 if the notion of `stable morphism' can be defined appropriately for our objects.
(See Remark~1.2 below for more detailed explanations.)
Once that is achieved in such a way that the corresponding moduli space
 of stable morphisms in our contents behaves good enough to have
 a reasonable tangent-obstruction theory or its extension
  as long as any kind of intersection theory
   or theory of constructible functions can apply,
  then one can define and compute in good cases
  (a version of) `{\it D-string world-sheet instanton numbers}'
 exactly as the moduli stack of stable maps in Gromov-Witten theory
  did for the fundamental closed or open string.
This is the topic of D(10).
It turns out that `{\it stable D-branes}' in string-theory contents
   as referred to whether they can decay or not
   has to do with  the notion of {\it (BPS) central charges of D-branes}.
Furthermore, besides this crucial notion from superstring theory,
 when one allows the topology of a D-string world-sheet to vary
  under deformations,
additional mathematical issues come in if one wants to obtain
 a compact moduli stack of stable morphisms.
The latter issue forces us to separate the D-string world-sheet
 to a major subcurve and and a minor subcurve
 in such a way that  the former with the fundamental module in general
  is pushed forward under the morphism to
  a  $1$-dimensional coherent sheaf on the target Calabi-Yau space $Y$
 while the latter consists of a collection of ${\Bbb P^1}$-trees
  and, with the fundamental module, is pushed forward
  only to a $0$-dimensional coherent sheaf on $Y$.
The former carries the main information of central charge
 while the latter is to be constrained by hand
 by requiring good properties on the induced morphism
 of related moduli stacks.
(See [L-Y5] for details).
This leads us to construct first an {\it auxiliary moduli stack},
 i.e.\ {\it the moduli stack of semistable (with respect to a central charge)
 Fourier-Mukai transforms of dimension $1$ and width $[0]$
  from a compact family of nodal curves to the Calabi-Yau manifold $Y$}.
These two sub-topics,
  central charges and the auxiliary moduli stack,
 are the focus of this note D(10.1).
In particular, we prove that such stacks are compact and, hence,
  provide us with a collection of reasonable reference stacks to begin with,
(Sec.~3, Theorem~3.1).


\bigskip

\begin{sremark}  {$[$where is the special connection?$]$.}  {\rm
 D-brane (or M-brane) instantons are generally more complicated
  than (fundamental) open or closed string world-sheet instantons,
  since a D-brane or M-brane carries more structure thereupon.
 The most basic such structure  for D-branes (in the simplest case)
   is an (open-string-induced ) bundle with a gauge connection.
 Besides the D-brane world-volume instantons addressed above,
   a D-brane of dimension $\ge 3$ (i.e.\ world-volume of dimension $\ge 4$ )
   can wrap around a cycle of an internal Calabi-Yau $3$-space
   to create another copy of $d=4$ effective space-time.
 Dimension reduction of the full (open-string-induced) field theory
   on the D-brane world-volume
   thus creates another sector to the complete $4$-dimensional effect field theory
   from the compactification of a $10$-dimensional Type II superstring theory.
 Quantity from this section in general may be subject to
   open-string world-sheet instanton corrections.
 This sector contains, in particular, a $4$-dimensional gauge theory,
  which may itself has  instanton solutions (in the sense of a gauge theory).
 All these different types of $4$-dimensional instantons arising from D-branes
 say the importance of special gauge connections in the definition of
  D-brane instantons in whatever contents.
 Thus, alert string-theorists may legitimately question
   our algebro-geometric inclined formalism:
   \begin{itemize}
    \item[{\bf Q.}]  \parbox[t]{37em}{\it
	 One has the bundle ${\cal E}$ on the D-string world-sheet,   but
	 where is the connection on ${\cal E}$ in this setting
	  that, as for any supersymmetric D-brane,
	     must satisfy supersymmetry-induced equations of motion?}
   \end{itemize}
 To answer this, one has to reason as follows.
  In our data, we specify only holomorphic structures on ${\cal E}$.
 In general there are nonunique connections on ${\cal E}$
  that are compatible with the given holomorphic structure on ${\cal E}$.
  {\it Stability condition} on the morphism comes in
    to distinguish string-theory-allowed ${\cal E}$
	while, at  the same time, select a unique compatible connection on it
	through a  Donaldson-Uhlenbeck-Yau-type Theorem (cf.\ [Don], [Ja], [Le], [U-Y]),
  though the latter cannot be spelled out literally yet in most cases.
 Thus, we have the following correspondence

 \bigskip

 {\small
  $$
   \begin{array}{cclccl}
     \mbox{\boldmath $\cdot$}
	  & \parbox[t]{14em}{Azumaya curve\\[.4ex] with a fundamental module\\[.6ex]
	                                          $(C,{\cal O}_C^{A\!z}; {\cal E})$   }
	     &&&&  \parbox[t]{21em}{(Euclidean)  D-string {\it world-sheet}\\[.4ex]
                                     		             with open-string-induced {\it structures}} \\[9ex]
	\mbox{\boldmath $\cdot$}
      & \parbox[t]{14em}{morphism to Calabi-Yau space\\[.6ex]
            $\varphi:(C,{\cal O}_C^{A\!z}; {\cal E})
			    \longrightarrow Y$ }
		&&   \Longleftrightarrow
        && \parbox[t]{21em}{{\it wrapping} of (Euclidean)
		           D-string world-sheet\\[.4ex] on a holomorphic $1$-cycle in $Y$}    \\[7ex]
	\mbox{\boldmath $\cdot$}
	 & \parbox[t]{14em}{stability condition on $\varphi$}
	    &&&&  \parbox[t]{21em}{distinguished holomorphic bundle\\[.4ex]
		              with a {\it unique special connection} thereupon\\[.4ex]  on D-string world-sheet}\,,
                                                    														                      \\[9ex]
   \end{array}
   $$
   }

   \noindent
   which would then justify a stable morphism in our content as a D-string world-sheet instanton.
 This is the same working philosophy behind many literatures on D-branes of B-type.
 }\end{sremark}

\bigskip

\begin{sremark}
{$[$effect of $B$-field$\,]$.}   {\rm
 Alert readers who bridge well between mathematics and string theory
   will notice immediately an incompleteness of our setting in this note:
 While we take into account the effect  of a $B$-field to the central charge
     of D-branes,
   we completely ignore its effect to the twisting of the Chan-Paton sheaf
     on the D-brane.
 With the language of {\it twisted sheaves}
    (cf.\ [C\u{a}] of Andrei C\u{a}ld\u{a}raru),
  D-branes in a space-time with a $B$-field background can be formulated
   as a morphism from a general Azumaya space $(X,{\cal  O}_X^{A\!z})$
     -- whose associated class in the Brauer group $\Br(X)$ of $X$ is non-zero --
	 with a compatible {\it twisted fundamental
	 ${\cal O}_X^{A\!z}$-module} $ {\cal E}$ ([L-Y2] (D(5)));  and
 the above two effects to D-branes (in particular, to D-string world-sheet instantons)
   can then be taken care of simultaneously.
 For the current note, we follow
  the TASI 2003 Lecture Notes of Paul Aspinwall [As]
   to take the former effect into account but suppress the latter for simplicity.
 This definitely leaves room for future works toward a complete treatment.
}\end{sremark}

\bigskip

\bigskip

\section{Twisted central charges and stability conditions on Fourier-Mukai transforms
                  from nodal curves}

In this section
 we define the stability conditions associated to central charges we will use in our problem,
 and study their basic properties and the chamber structure in the space of stability conditions.
A conjecture is made  on the wall-crossing behavior of the related moduli space.

\bigskip
		
\subsection{Twisted central charges of D-string world-sheet instantons}

We recall in this subsection
 the origin of (BPS) central charge of a D-brane in a Calabi-Yau $3$-fold
  in superstring theory
 and then explain why we need a twist of it in our problem.
After that,
 we define precisely the central charge functional to be used for our problem
 and study its basic properties.

\bigskip

\begin{flushleft}
{\bf Central charges of D-branes in a Calabi-Yau $3$-fold and why we need a twist.}
\end{flushleft}
The $d=4$, $N=2$ supersymmetry algebra  ${\cal A}$ (over ${\Bbb C}$)
 has a $1$-dimensional center.  Let $Z$ be a generator of this center.
Then, for an irreducible representation ${\cal H}$  of ${\cal A}$,
 $Z$ acts on ${\cal H}$  as $c\cdot \Id_{\cal H}$,
   where $c\in {\Bbb C}$ is a constant  and
              $\Id_{\cal H}$ is the identity map on ${\cal H}$.
$c$ is called the {\it central charge of the representation}.
We also say that the elements in the representation ${\cal H}$
 carry a central charge $c$.
When the Type IIA or the Type IIB $10$-dimensional superstring theory
 is compactified on a general Calabi-Yau $3$-fold $Y$,
 i.e.\ the whole space-time becomes now ${\Bbb R}^{3+1}\times Y$,
the resulting effective field theory on the $4$-dimensional Minkowski space-time
 ${\Bbb R}^{3+1}$ has a $d=4$, $N=2$ supersymmetry
 ${\cal A}_{d=4,N=2}$.
When one adds a D-brane $X$ (with structures thereupon kept implicit)
 to the internal Calabi-Yau $3$-fold $Y$ over a $p\in {\Bbb R}^{3+1}$,
$X$ would be effectively a pointlike particle sitting at $p$
 from the ${\Bbb  R}^{3+1}$ aspect.
When $X$ in ${\Bbb R}^{3+1}\times Y$
 evolves along with time, we obtain then an embedding of the world-volume
 $X\times {\Bbb R}$ of the D-brane in ${\Bbb R}^{3+1}\times Y$.
{From} the effective $4$-dimension aspect, this is an embedding
 ${\Bbb R}^{0+1}\hookrightarrow {\Bbb R}^{3+1}$
 of the world-line of a particle in the Minkowski space-time.
In general, such a configuration in a compactification of a superstring theory
 renders no supersymmetry left in the $4$-dimensional effective theory.
For a special D-brane (i.e.\ those wrapping a calibrated or holomorphic cycle,
   as is in our case),
 this will reduce the previous $d=4$, $N=2$ supersymmetry only to
 $d=4$, $N=1$ supersymmetry in the $4$-dimensional effective theory
 and the corresponding particle is thus a {\it BPS particle}, by definition.
In terms of quantum mechanics of a particle in ${\Bbb R}^{3+1}$,
 it corresponds to a state in an irreducible representation ${\cal H}$ of
 ${\cal A}_{d=4,N=2}$ whose annihilator ${\cal A}_0$
  is a $d=4$, $N=1$ subalgebra of ${\cal A}_{d=4,N=2}$.
The central charge $c$ of ${\cal H}$ (or equivalently this BPS state)
 is defined to be the {\it central charge of the D-brane} $X$
 in the Calabi-Yau $3$-fold $Y$. $\; c$ determines ${\cal A}_0$.
This is the $4$-dimensional effective space-time aspect
 of the central charge of  a D-brane $X$ in $Y$.

{From} the superstring world-sheet aspect,
  the nonlinear $\sigma$-model on a closed string world-sheet $\Sigma_{closed}$
  with target a general Calabi-Yau $3$-fold $Y$
  defines a $d=2$, $N=(2,2)$ superconformal field theory
  on $\Sigma_{closed}$.
When one adds a D-brane $X$ to $Y$ and requires the open-string
 to have their end-points lying in $X$,
the resulting $d=2$ open-string world-sheet theory in general has no longer
  any supersymmetry.
Again, for a special D-brane $X$,
 the nonlinear $\sigma$-model with target $(Y,X)$ gives
 a $d=2$, $N=(1,1)$ superconformal field theory
 on the open string world-sheet
 $(\Sigma_{open},\partial\Sigma_{open})$.
As an abstract open-and-closed conformal field theory,
 a $d=2$, $N=(1,1)$ superconformal field theory with boundary
 is obtained from a $d=2$, $N=(2,2)$ superconformal field theory
 by specifying how its holomorphic sector and its antiholomorphic sector
  should be matched along the boundary $\partial\Sigma_{open}$.
The rule is specified exactly by the D-brane $X$
 (i.e.\ a boundary condition to the $d=2$ theory) and
the phase of the central charge of $X$ enters the rule.
With more technicality, the central charge of the D-brane $X\subset  Y$
  can be reproduced from the expansion of the boundary state
   $|B_X\rangle$ associated to $X$ in terms of  Ishibashi states
   in the $d=2$ superconformal field theory.
This gives a meaning of the (BPS) central charge of a D-brane
 from the open-string world-sheet aspect.

Mathematicians are referred to
 [Ar], [A-B-C-D-G-K-M-S-S-W], [B-B-S], [H-K-K-P-T-V-V-Z], and [O-O-Y]
 for more detailed explanations.

\bigskip

Back to the mathematical world,
Given
 a Calabi-Yau $3$-fold $Y$ with a complexified K\"{a}hler class $B+\sqrt{-1}J$
 and a D-brane of B-type, realized as a coherent sheaf ${\cal F}$
 (or more generally an object in the derived category $D^b(\Coh(Y))$
    of coherent sheaves) on $Y$,
 after several string-theorists' works,
   e.g.\ [As], [C-Y], [F-W], [G-H-M], [Harv], [M-M], [O-O-Y],
 the formula of the central charge of D-branes in this case
 is given by
 $$
   Z^{B+\sqrt{-1}J}({\cal F})\;
     =\;    \int_Y  e^{-(B+\sqrt{-1}J)}
	                        \ch({\cal F})\sqrt{\td(T_Y)}\;
			 +\; O(\alpha^{\prime}) \,,
 $$
 where $O(\alpha^{\prime})$ is a stringy correction to $Z({\cal F})$
   that tends to zero as $\alpha^{\prime}\rightarrow 0$
   (i.e.\ fundamental string tension $l_s\rightarrow \infty$).
Naively, in our situation we have a coherent $\tilde{\cal F}$  on $C\times Y$
 and one would like to define the central charge of $\tilde{\cal F }$ to be
 $$
   Z^{B+\sqrt{-1}J}_{\scriptsizenaive}({\cal F})\;
     =\;    \int_{C\times Y}
	              \pr_2^{\ast}
                          \left(
					         \frac{e^{-(B+\sqrt{-1}J)}}
                                    {\sqrt{\td(T_Y)}}
                          \right)
                  \tau_{C\times Y}(\tilde{\cal F})\;
			 +\; O(\alpha^{\prime}) \,.
 $$
 When $\Supp(\tilde{\cal F})$ defines an embedding
  $\iota: C\hookrightarrow Y$,
 then
  $Z^{B+\sqrt{-1}J}_{\scriptsizenaive}(\tilde{\cal F})
   = Z^{B+\sqrt{-1}J} (\iota_{\ast}{\cal E})$,
  where ${\cal E}:= \pr_{1,\ast}\,\tilde{\cal F}$.
However,
  e.g.\ by considering the case $Y= S \times E$,
      where $S$ is an algebraic $\Kthree$-surface
        	  that contains an embedded ${\Bbb P^1 }$ and
    $E$ is a smooth curve of genus $1$,
it can happen that under a deformation of $\tilde{\cal F}$,
 $\tilde{\cal F}$ is deformed to
 $\tilde{\cal F}^{\prime}
   = \tilde{\cal F}^{\prime}_1 + \tilde{\cal F}^{\prime}_2$
   such that
         $\Supp(\tilde{\cal F}^{\prime}_1)
		     \cap \Supp(\tilde{\cal F}^{\prime}_2) = \emptyset$
       and that $\pr_{2\ast}\,\tilde{\cal F}^{\prime}_2$
	                 is $0$-dimensional.
In this case, any coherent sheaf on $\Supp(\tilde{\cal F}^{\prime}_2)$
 is $Z^{B+\sqrt{-1}J}_{\scriptsizenaive}$-semistable
and, hence, the moduli stack
 of $1$-dimensional $Z^{B+\sqrt{-1}J}_{\scriptsizenaive}$-semistable
 coherent sheaves on $C\times Y$ of the fixed central charge
 won't be bounded in general.
To cure this, one needs to recover the positivity of $-\Imaginary Z$
 by introducing a twist from a positive degree class on $C$.

\bigskip

\begin{flushleft}
{\bf Twisted central charges of D-string world-sheet instantons.}
\end{flushleft}
Let
  $C$ be a nodal curve with a polarization class $L$  and
  $Y$ be a projective Calabi-Yau manifold with a complexified K\"{a}hler class
    $B+\sqrt{-1}J$.
	
\bigskip
				
\begin{definition}
{\bf [twisted central charge of Fourier-Mukai transform].} {\rm
 Let
   $\tilde{\cal F}$ be a coherent sheaf of dimension $1$ on $C\times Y$   and
   $\Phi_{\tilde{\cal F}}$ be the Fourier-Mukai transform $\tilde{\cal F}$ defines.
 Then, the {\it twisted central charge} of $\Phi_{\tilde{\cal F}}$
  associated to the data $(B+\sqrt{-1}J,L)$ is defined to be
  $$
   Z^{B+\sqrt{-1}J, L}(\Phi_{\tilde{\cal F}})\;
   :=\; Z^{B+\sqrt{-1}J, L}(\tilde{\cal F)}\;
   :=\;   \int_{C\times Y}\,
                \pr_2^{\ast}\left(
                    \frac{e^{-(B+\sqrt{-1}J)}}{\sqrt{td(T_Y)}}
                                          \right)\,
                \cdot\, \pr_1^{\ast}\, e^{-\sqrt{-1}L}\,
                \cdot\, \tau_{C\times Y}(\tilde{\cal F})\,,
  $$
  where  $\tau_{C\times Y}(\tilde{\cal F}):= \ch(\tilde{\cal F })
                    \cdot \td(T_{C\times Y})$ is the $\tau$-class of $\tilde{\cal F}$.
} \end{definition}		

\bigskip

\begin{lemma}
 {\bf [twisted central charge: explicit form].}
  Continuing the above notation.
  Let
   $$
     \tilde{\beta}(\tilde{\cal F})\; :=\;  \sum_i d_i[\zeta_i]\; \in\;  A_1(C\times Y)\,,
   $$
    where
	  $\zeta_i$ runs through the generic points of $\Supp(\tilde{\cal F})$   and
      $d_i$  is the dimension of $\tilde{\cal F}|_{\zeta_i}$
		   as a $k_{\zeta_i}$-vector space.
  Then,
   $$
     Z^{B+\sqrt{-1}J,L}(\tilde{\cal F})\;
      =\;  \left(\chi(\tilde{\cal F})-B\cdot\tilde{\beta}(\tilde{\cal F})\right)\,
 	        -\,\sqrt{-1}\,\left((J+L)\cdot \tilde{\beta}(\tilde{\cal F})\right)\,.
   $$
  In particular, for non-zero coherent sheaves on $C\times Y$ of dimension $\le 1$,
   $Z^{B+\sqrt{-1}J,L}$ takes its values
    in the partially completed lower-half complex plane
   $$
     \hat{\Bbb H}_-\;
	 :=\; \{ z\in {\Bbb C}\,|\,
                   \mbox{either $\Imaginary z <0$
		                          or $\Imaginary z=0$ with $\Real z >0$}\, \}\,.
   $$
\end{lemma}

\begin{proof}
 The design of the functorial $\tau$-class $\tau(\,\cdot\,)$ for coherent sheaves
   on singular varieties to fit well with the Grothendieck-Riemann-Roch Theorem and
 the behavior of $\tau$-class in a flat family imply that
 $$
  \tau_{C\times Y}(\tilde{\cal F})\; =\;
    \tilde{\beta}(\tilde{\cal F})+ Z_0
 $$
 with $Z_0\in A_0(C\times Y)$ of degree $\chi(\tilde{\cal F})$.
 The lemma follows.
\end{proof}

\bigskip

\begin{remark}{\rm
 {\it $[$general case and stringy correction$]$.}
 Though our main interest motivated from superstring theory is the case $Y$
   is a Calabi-Yau $3$-fold.
 For a general $Y$ with $c_1(Y)\ne 0 $ and of general dimensions,
  the expression for $Z^{B+\sqrt{-1}J,L}$ in Lemma~2.1.2
  is modified to
  $$
    Z^{B+\sqrt{-1}J,L}(\tilde{\cal F})\;
     =\;  \left(
	           \chi(\tilde{\cal F})
	           -(B+\frac{1}{4}c_1(Y))\cdot\tilde{\beta}(\tilde{\cal F})
		    \right)\,
 	      -\,\sqrt{-1}\,\left((J+L)\cdot \tilde{\beta}(\tilde{\cal F})\right)\,.
  $$
 Purely mathematically, this may be taken as the starting point of the current note.
 All the statements remain to hold after being mildly modified accordingly if necessary.
 Furthermore, when the stringy correction $O(\alpha^{\prime})$ to the central charge
  $Z^{B+\sqrt{-1}J,L}$  is taken into account,
  the evaluation of differential $2$-forms on (real $2$-)cycles in the above expression
  should be replaced by a quantum evaluation from Gromov-Witten theory.
}\end{remark}

\bigskip

\begin{definition}
 {\bf [$Z$-slope $\mu^Z$].} {\rm
 Continuing Definition~2.1.1.
 We define the {\it $Z$-slope} for a non-zero coherent sheaf $\tilde{\cal F}$ on $C\times Y$
  of dimension $\le 1$ to be
  $$
   \begin{array}{crl}
      \mu^Z(\tilde{\cal F})
       &  :=
	   &  - \left.
	          \Real \left(Z^{B+\sqrt{-1}J,L}(\tilde{\cal F})\right)
	          \right /
		      \Imaginary\left(Z^{B+\sqrt{-1}J,L}(\tilde{\cal F})\right)   \\[1.6ex]
     & = 	
	   & \left.
      	     \left(\chi(\tilde{\cal F})-B\cdot\tilde{\beta}(\tilde{\cal F})\right)
		     \right/
 	         \left((J+L)\cdot \tilde{\beta}(\tilde{\cal F})\right)\,.
   \end{array}
  $$
}\end{definition}

\bigskip

\begin{flushleft}
{\bf Basic properties.}
\end{flushleft}
We collect here a few basic properties of a central charge functional $Z^{B+\sqrt{-1}J}$
 we need for later discussions.

\bigskip

\begin{lemma}
{\bf [invariant of flat family].}
 Let
  $(C_S,L_S)$ be a flat family of nodal curves with a polarization class over $S$
   and
  $\tilde{\cal F}_S$ be a coherent sheaf on $C_S\times Y$
    that is flat and of relative dimension $1$ over $S$.
 Then $Z^{B+\sqrt{-1}J, L_s}(\tilde{\cal F}_s)\in {\Bbb C}$, $s\in S$,
   is locally constant on $S$. 	
\end{lemma}

\begin{proof}
 By considering the restriction of the coherent sheaf $\tilde{\cal F}_S$ to generic points
  of $\Supp(\tilde{\cal F}_S)$,
 one observes that, under flat deformations of $(C,L, \tilde{\cal F})$,
   $\tilde{\beta}(\,\bullet\,)$  form a flat family of $1$-cycles on the family $(C_S\times Y)/S$ .
 Since, in addition, $\chi(\,\bullet\,)$  are constant, the lemma follows.

\end{proof}

\bigskip

\begin{lemma}
 {\bf [additivity].}
  Let
    $0 \rightarrow  \tilde{\cal F}_1 \rightarrow \tilde{\cal F}_2
	     \rightarrow  \tilde{\cal F}_3 \rightarrow 0$
   be a short exact sequence of coherent sheaves on $C\times Y$ of dimension $\le 1$.
 Then,
   $$
     Z^{B+\sqrt{-1}J,L}(\tilde{\cal F}_2)\;
      =\;  Z^{B+\sqrt{-1}J,L}(\tilde{\cal F}_1)\,
			 +\,  Z^{B+\sqrt{-1}J,L}(\tilde{\cal F}_3)\,.
   $$
\end{lemma}

\begin{proof}
 The $\tau$-class is additive with respect to a short exact sequence.
 In the explicit form,
  both $\tilde{\beta}(\,\bullet\,)$ and $\chi(\,\bullet\,)$
  are additive with respect to a short exact sequence.

\end{proof}

\bigskip

\begin{lemma}
 {\bf [finite set of corresponding Hilbert polynomials].}
  Given
      a bounded flat family $(C_S,L_S)$ of nodal curves with a polarization class over $S$,
  let
    $c\in  \hat{\Bbb H}_-$,
     $H$ be a relative hyperplane class on $(C_S\times Y)/S$  and
	 $ \Poly^H(Z^{B+\sqrt{-1}J,L}=c)$
	  be the set of Hilbert polynomials of coherent sheaves $\tilde{\cal  F}$ of dimension $\le 1$
	     on fibers of $(C_S\times Y)/S$
       	 with $Z^{B+\sqrt{-1}J,L}(\tilde{\cal F})=c$.
 Then,  $ \Poly^H(Z^{B+\sqrt{-1}J,L}=c)$  is a finite set.
\end{lemma}

\begin{proof}
 Recall that the Hilbert polynomial defined by $H$(in variable $m$)
  of a coherent sheaf on fibers of $(C_S\times Y)/S$ of dimension $\le 1$ is given by
   $$
	  P^H(\tilde{\cal F})\;
       =\;  (H\cdot\tilde{\beta}(\tilde{\cal F}))m+\chi(\tilde{\cal F})\,.
   $$
 Let $\iota_S: C_S\times Y\hookrightarrow{\Bbb P}^N_S$ be an $S$-embedding
   of $C_S\times Y$ into the projective-space bundle ${\Bbb P}^N_S$ over $S$
   determined by $H$.
 Then
   $N_1({\Bbb P}^N_S/S)_{\Bbb Z}\simeq {\Bbb P}^N_S\times {\Bbb Z}$
    is a trivial bundle of free abelian group of rank $1$, equipped with a canonical trivialization
	from the generating effective curve class in fibers of ${\Bbb P}^N_S/S$.
 Let
  $\underline{\iota}_{S,\ast}:N_1((C_S\times Y)/S)_{\Bbb Z}\rightarrow {\Bbb Z}$
     be the composition
     $N_1((C_S\times Y)/S)_{\Bbb Z}\rightarrow N_1({\Bbb P}^N_S/S)_{\Bbb Z}
      \rightarrow 	{\Bbb Z}$
	of the induced bundle homomorphism  and the projection map.
 Consider the subset of effective curve classes in fibers of $(C_S\times Y)/S$
   $$
    \Xi_ {(C_S\times Y)/S, 1}^{J+L,c}\ :=\;
	     \left\{
	        E\in \overline{\NE}_1((C_S\times Y)/S)_{\Bbb Z}\;|\;
				  (J+L)\cdot E\,=\, -\Imaginary c
		 \right\}
   $$
 Then, it follows from Kleiman's Ampleness Criterion that
    $\underline{\iota}_{S,\ast}\big(\Xi_{(C_S\times Y)/S, 1}^{J+L, c}\big)$
 	is a finite subset of ${\Bbb Z}$.
 Consequently,
  both the set
    $$
	  \left \{H\cdot \tilde{\beta}(\tilde{\cal F})\,
	     \left|\,
          \parbox{20em}
		  {$\tilde{\cal F}$ is a coherent sheaf on fibers of $(C_S\times Y)/S$ \\
	             of dimension $\le 1$ with $Z^{B+\sqrt{-1}J,L}(\tilde{\cal F})=c$}
		\right.
	  \right\}\; \subset\;  {\Bbb Z}
	$$
	and the set
    $$
	  \left \{\pr_{2,\ast}(\tilde{\beta}(\tilde{\cal F}))\,
	     \left|\,
          \parbox{20em}
		  {$\tilde{\cal F}$ is a coherent sheaf on fibers of $(C_S\times Y)/S$ \\
	             of dimension $\le 1$ with $Z^{B+\sqrt{-1}J,L}(\tilde{\cal F})=c$}
		\right.
	  \right\} \;
	  \subset\;  N_1(Y)_{\Bbb Z}
	$$
   are finite.
 It follows that
 the set
   $$
	  \left \{B\cdot \tilde{\beta}(\tilde{\cal F})\,
	     \left|\,
          \parbox{20em}
		  {$\tilde{\cal F}$ is a coherent sheaf on fibers of $(C_S\times Y)/S$ \\
	             of dimension $\le 1$ with $Z^{B+\sqrt{-1}J,L}(\tilde{\cal F})=c$}
		\right.
	  \right\} \;
	  \subset\;  {\Bbb R}
	$$
 and, hence, the set
     $$
	  \left \{\chi(\tilde{\cal F})\,
	     \left|\,
          \parbox{20em}
		  {$\tilde{\cal F}$ is a coherent sheaf on fibers of $(C_S\times Y)/S$ \\
	             of dimension $\le 1$ with $Z^{B+\sqrt{-1}J,L}(\tilde{\cal F})=c$}
		\right.
	  \right\} \;
	  \subset\;  {\Bbb Z}
	$$
  are also finite.
 This proves the lemma.

\end{proof}

\bigskip

The following lemma relates $\mu^Z$ with the more familiar $\mu^P$:

\bigskip

\begin{lemma}
 {\bf [$P$-slope vs.\ $Z$-slope].}
  Let
   $L$ be a relative polarization class on the universal curve
     $C_{\cal  M}/{\cal M}$ and
   $H$ be a relative hyperplane class
      on $(C_{\cal M}\times Y)/{\cal M}$.
  Then
   there exist constants $c_1,\; c_2,\, c_3,\, c_4\in {\Bbb R}$,  with $c_1,\,c_3>0$,
    that depend only on $B+\sqrt{-1}J$, $L$, and $H$
  such that
   $$
     c_1\, \mu^P (\tilde{\cal F})+c_2\;
	   \le\;  \mu^Z(\tilde{\cal F})\;
	   \le \; c_3\,\mu^P(\tilde{\cal F}) + c_4\,
  $$	
  for all nodal curves $C$ in the family $C_{\cal M}/{\cal M}$ and
       all $1$-dimensional coherent sheaves $\tilde{\cal F}$ on $C\times Y$.
  Here
   $$
     P(\tilde{\cal F})\;
	  =\; (H\cdot \tilde{\beta}(\tilde{\cal F}))m \,+\, \chi(\tilde{\cal F})
   $$
       is the Hilbert polynomial (in $m$)of $\tilde{\cal F}$ defined by $H$,
   $$
     \mu^P (\tilde{\cal F})\;
	  :=\;  \chi(\tilde{\cal  F})/(H\cdot \tilde{\beta}(\tilde{\cal F}))
   $$
     is the slope of $\tilde{\cal  F}$ defined by $P$.
\end{lemma}

\begin{proof}
 Let $S$ be an atlas of ${\cal M}$.
 Under the same setting as in the proof of Lemma~2.1.7,
   $$
    \frac{B\cdot(\,\bullet\,)}{(J+L)\cdot(\,\bullet\,)}\;:\;
	  {\Bbb P}(\overline{\NE}_1((C_S\times Y)/S))\; \longrightarrow\; {\Bbb R}
   $$
   has image in a bounded subset of ${\Bbb R}$
  while
   $$
    \frac{H\cdot(\,\bullet\,)}{(J+L)\cdot(\,\bullet\,)}\;:\;
	 {\Bbb P}(\overline{\NE}_1((C_S\times Y)/S))\; \longrightarrow\; {\Bbb R}_{>0}
   $$
   has image in a compact subset of ${\Bbb R}_{>0}$.
 The lemma follows.

\end{proof}

\bigskip

\begin{proposition}
 {\bf [properness/projectivity of Quot-scheme].}
 Given
   a family of nodal curves $C_S/S$ over $S$ with a relative polarization class $L$   and
   a $1$-dimensional coherent sheaf $\tilde{\cal F}_S$ on $(C_S\times Y)/S$.
 Then
  the $\Quot$-scheme
     $\Quot^{Z^{B+\sqrt{-1},L}}_{(C_S\times Y)/S}
            (\tilde{\cal F}_S,c  )$
  of quotient sheaves $\tilde{\cal F}_S\rightarrow \tilde{\cal Q}_S\rightarrow 0$
  that is flat over $S$ with $Z^{B+\sqrt{-1}J,L}(\tilde{\cal Q}_S)=c$
  is projective over $S$.
\end{proposition}

\begin{proof}
 Let $H$ be a relative ample line bundle on $(C_S\times Y)/S$.
 Since both the Hilbert polynomial $P^H$ and the central charge $Z^{B+\sqrt{-1}J,L}$
  are invariant under flat deformations,
 connected components of
  $\Quot^{Z^{B+\sqrt{-1},L}}_{(C_S\times Y)/S}(\tilde{\cal F}_S,c  )$
 coincide with some connected components
  of $\Quot^H_{(C_S\times Y)/S}(\tilde{\cal F}_S)$.
 It follows from Lemma~2.1.7
  that
  $$
    \Quot^{Z^{B+\sqrt{-1},L}}_{(C_S\times Y)/S}
            (\tilde{\cal F}_S,c  )\;
	=\;  \coprod_{P_i\, \in\,  \scriptsizePoly^H(Z^{B+\sqrt{-1}J,L}=c) }\,
	          \Quot^H_{(C_S\times Y)/S}(\tilde{\cal F}_S, P_i)
  $$
  is a finite disjoint union.
 The latter is known to be projective over $S$.

\end{proof}

\bigskip

\bigskip
	
\subsection{Stability conditions on Fourier-Mukai transforms
        of dimension $1$ and width $[0]$ from nodal curves}

Recall
 the compact family $C_{\cal M}/{\cal M}$ of nodal curves
  with a relative ample class $L$ and
 the projective Calabi-Yau $3$-fold $Y$
    with a complexified K\"{a}hler class $B+\sqrt{-1}J$.
We introduce
   the notion of $Z$-(semi)stability of Fourier-Mukai transforms
     of dimension $1$ and width $[0]$
    from a fiber of $C_{\cal M}/{\cal M}$ to $Y$  and
   their Harder-Narasimhan filtration with respect to $Z$,    and
 collect their basic properties in this subsection.
Though the technique here remains pre-Bridgeland,
 readers are recommended
  to [Br] of Tom Bridgeland for the generalization to objects
     in the derived category of coherent sheaves  and
  to [Dou2] and [Dou3] of Michael Douglas
    for the string-theoretical motivation and reason of the design.

\bigskip

\begin{flushleft}
{\bf  Stability conditions defined by a central charge functional $Z$.}
\end{flushleft}
\begin{definition}
 {\bf [$Z$-semistable, $Z$-stable, $Z$-unstable, strictly $Z$-semistable].} {\rm
 Let $[C]\in {\cal M}$.
 A $1$-dimensional coherent sheaf $\tilde{\cal F}$ on $C\times Y$
   is said to be {\it $Z$-semistable} (resp.\ $Z$-stable)
  if $\tilde{\cal F}$ is pure and
     $\mu^Z(\tilde{\cal F}^{\prime})\le $ (resp.\ $<$)
     $\mu^Z(\tilde{\cal F})$
	 for any nonzero proper subsheaf
	   $\tilde{\cal F }^{\prime}\subset \tilde{\cal F}$.
 Such $\tilde{\cal F}$ is called $Z$-{\it unstable} if it is not $Z$-semistable,   and
  is called {\it strictly $Z$-semistable} if it is $Z$-semistable but not $Z$-stable.
 When the central charge functional  $Z$ is known and fixed either explicitly or implicitly,
  we may use the terminology: {\it semistable, stable, unstable, strictly semistable},
  for simplicity.
}\end{definition}

\bigskip

\begin{notation}
 {\bf  [$Z$-(semi)stable].} {\rm (Cf.\ [H-L: Notation~1.2.5].) } {\rm
  Following Huybrechts and Lehn [H-L], we'll rephrase, for example,
   the above definition in a combined statement and notation
  ``{\it $\tilde{\cal  F}$ is (semi)stable
         if it is pure and
		    $\mu^Z(\tilde{\cal F}^{\prime})\,(\le)\, \mu^Z(\tilde{\cal F})$
		    for any nonzero proper subsheaf
			       $\tilde{\cal F}^{\prime}\subset \tilde{\cal F}$ }."
  to cover both the $Z$-semistable and the $Z$-stable situation.
}\end{notation}

The following statements from [H-L: Sec.~1.2] remain to be true by the same argument.

\bigskip

\begin{proposition}
{\bf [equivalent form].}   {\rm (Cf.\ [H-L: Proposition~1.2.6].)}
 Let $\tilde{\cal E}$ be a purely $1$-dimensional coherent sheaf on $C\times Y$.
 Then the following statements are equivalent:
 \begin{itemize}
  \item[(1)]
   $\tilde{\cal F}$ is (semi)stable.

  \item[(2)]
   $\mu^Z(\tilde{\cal F}^{\prime})\,(\le)\,\mu^Z(\tilde{\cal F})$
     for all nonzero proper saturated subsheaves
          $\tilde{\cal F}^{\prime}\subset \tilde{\cal F}$.
	
  \item[(3)]
   $\mu^Z(\tilde{\cal F})\,(\le)\, \mu^Z(\tilde{\cal F}^{\prime\prime})$
   for all nonzero proper $1$-dimensional quotient sheaves
   $\tilde{\cal F}\rightarrow \tilde{\cal F}^{\prime\prime}$.

  \item[(4)]
   $\mu^Z(\tilde{\cal F})\,(\le)\, \mu^Z(\tilde{\cal F}^{\prime\prime})$
   for all nonzero proper purely $1$-dimensional quotient sheaves
   $\tilde{\cal F}\rightarrow \tilde{\cal F}^{\prime\prime}$.
 \end{itemize}
\end{proposition}

\bigskip

\begin{proposition}
 {\bf [$Z$-slope and homomorphism].}  {\rm (Cf.\ [H-L: Proposition 1.2.7].) }
  Let $\tilde{\cal F}$ and $\tilde{\cal G}$ be purely $1$-dimensional coherent sheaf
   on $C\times Y$.
  If $\mu^Z(\tilde{\cal F})> \mu^Z(\tilde{\cal G})$,
   then $\Hom(\tilde{\cal F},\tilde{\cal G})=0$.
  If $\mu^Z(\tilde{\cal F})=\mu^Z(\tilde{\cal G})$ and
   $h: \tilde{\cal F}\rightarrow \tilde{\cal G}$ is nonzero,
  then $h$ is injective if $\tilde{\cal F}$ is $Z$-stable   and
                     surjective if $\tilde{\cal G}$ is $Z$-stable.
  If $Z(\tilde{\cal F})=Z(\tilde{\cal G})$,
   then any nonzero homomorphism
    $h:\tilde{\cal F}\rightarrow \tilde{\cal G}$ is an isomorphism
	provided $\tilde{\cal F}$ or $\tilde{\cal G}$ is $Z$-stable.
\end{proposition}

\bigskip

\begin{corollary}
 {\bf [$Z$-stable $\Rightarrow$ simple].}  {\rm (Cf.\ [H-L: Corollary 1.2.8].) }\\
  If $\tilde{\cal F}$ is a $Z$-stable sheaf,
   then $\End(\tilde{\cal F})\simeq {\Bbb C}$.
\end{corollary}

\bigskip

\bigskip

\begin{flushleft}
{\bf The Harder-Narasimhan filtration with respect to $Z$. }
\end{flushleft}

\begin{definition}
 {\bf [Harder-Narasimhan filtration].} {\rm
 Let $\tilde{\cal F}$ be a purely $1$-dimensional coherent sheaf on $C\times Y$.
 A {\it Harder-Narasimhan filtration} of $\tilde{\cal F}$
  (with respect to a central charge $Z$)
  is an increasing filtration
  $$
    0\;=\; \HN^Z_0(\tilde{\cal F})\; \subset \; \HN^Z_1(\tilde{\cal F})\;
	   \subset\; \cdots\; \subset \; \HN^Z_l(\tilde{\cal F})\; =\; \tilde{\cal F}\,,
  $$
  such that the factors
    $\gr_i^{HN^Z}
	   :=  \HN^Z_i(\tilde{\cal F})/\HN^Z_{i-1}(\tilde{\cal F})$,
	  for $i=1,\,\ldots\,,\, l$, are $Z$-semistable of dimension $1$
	  with $Z$-slopes $\mu_i$ satisfying
	$$
         \mu^Z_{max}(\tilde{\cal F})\;
		  :=\; \mu_1\; > \;\mu_2\; >\; \cdots\;
		  >\; \mu_l\; =:\; \mu^Z_{min}(\tilde{\cal F})\,.
    $$
}\end{definition}

\bigskip

\begin{lemma}
 {\bf [$\mu^Z_{min}$, $\mu^Z_{max}$,  and homomorphism].}
 {\rm (Cf.\ [H-L: Lemma~1.3.3].)}
 If $\tilde{\cal F}$ and $\tilde{\cal G}$
   are purely $1$-dimensional coherent sheaves on $C\times Y$
   with $\mu^Z_{min}(\tilde{\cal F})>\mu^Z_{max}(\tilde{\cal G})$,
  then $\Hom(\tilde{\cal F},\tilde{\cal G})=0$.
\end{lemma}

\bigskip

\begin{lemma}
 {\bf [subsheaf with maximal $Z$-slope].}  {\rm (Cf.\ [H-L: Lemma~1.3.5].)}
 Let $\tilde{\cal F}$ be a purely $1$-dimensional coherent sheaf on $C\times Y$.
 Then there exists a subsheaf $\tilde{\cal G}\subset \tilde{\cal F}$
  such that for all subsheaves $\tilde{\cal F}^{\prime}\subset \tilde{\cal F}$,
  one has $\mu^Z(\tilde{\cal F}^{\prime})\le \mu^Z(\tilde{\cal G})$,
  and in case of equality $\tilde{\cal F}^{\prime}\subset \tilde{\cal G}$.
 Furthermore, $\tilde{\cal G}$  is uniquely determined by $Z$ and is $Z$-semistable.
\end{lemma}

\bigskip

\begin{definition}
{\bf [maximal destabilizing subsheaf].}  {\rm
 Continuing Lemma~2.2.8.
 $\tilde{\cal G}$ is called
  the {\it maximal $Z$-destabilizing subsheaf} of $\tilde{\cal F}$.
}\end{definition}

\bigskip

\begin{theorem}
 {\bf [existence and uniqueness of Harder-Narasimhan filtration].}
 {\rm (Cf.\ [H-L: Theorem~1.3.4].)}
 Every purely $1$-dimensional coherent sheaf $\tilde{\cal F}$
  on $C\times Y$ has
  a unique Harder-Narasimhan filtration with respect to $Z$.
\end{theorem}

\bigskip

\begin{theorem}
 {\bf [Harder-Narasimhan filtration stable under base field extension].}
 {\rm (Cf.\ [H-L: Theorem~1.3.7].)}
  Let $\tilde{\cal F}$ be a purely $1$-dimensional coherent sheaf on $C\times Y$
   and $K$ be a field extension of $k \simeq  {\Bbb C}$.
  Then
   $$
     \HN_{\bullet}(\tilde{\cal F}\otimes_k K)\;
	 =\; \HN_{\bullet}(\tilde{\cal F})\otimes_k K\,.
   $$
\end{theorem}

\bigskip

\begin{corollary}
 {\bf [semistability under field extension].}
 {\rm (Cf.\ [H-L: Corollary~1.3.8].)}
  If $\tilde{\cal F}$ is a $Z$-semistable $1$-dimensional coherent sheaf
      on $C\times Y$   and
	$K$ is a field extension of $k\simeq {\Bbb C}$,
 then $\tilde{\cal F}\otimes_kK$ is $Z$-semistable as well.	
\end{corollary}

\clearpage

\bigskip

\begin{flushleft}
{\bf Jordan-H\"{o}lder filtrations and {\boldmath $S$}-equivalence with respect to $Z$. }
\end{flushleft}
\begin{definition}
{\bf [Jordan-H\"{o}lder filtration of semistable object].}  {\rm
 Let $\tilde{\cal F}$ be $Z$-semistable coherent sheaf of dimension $1$ on $C\times Y$.
 A {\it Jordan-H\"{o}lder filtration} of $\tilde{\cal F}$ with respect to $Z$
  is a filtration
  $$
    0\; =\; \tilde{\cal F}_0\; \subset \; \tilde{\cal F}_1\;
	     \subset\;  \cdots\; \subset\; \tilde{\cal F}_l\; =\; \tilde{\cal F}
  $$
  such that the factors
   $\gr_i(\tilde{\cal F}):= \tilde{\cal F}_i/\tilde{\cal F}_{i-1}$
   are $Z$-stable with $Z$-slope $\mu^Z(\tilde{\cal F})$.
}\end{definition}

\bigskip

\begin{proposition}
{\bf [existence of JH filtration/uniqueness of graded object].}
{\rm (Cf.\ [H-L: Proposition~1.5.2].)}
 Continuing Definition~2.2.13.
  Jordan-H\"{o}lder filtrations always exist.
  The graded object $\gr(\tilde{\cal F}):= \oplus_i\gr_i(\tilde{\cal F})$
    does not depend on the choice of the Jordan-H\"{o}lder filtration.
\end{proposition}

\bigskip

\begin{definition}
{\bf [$S$-equivalence of semistable objects].}  {\rm
 Two $Z$-semistable coherent sheaves
    $\tilde{\cal F}_1$ and $\tilde{\cal F}_2$ of dimension $1$ on $C\times Y$
    are called {\it S-equivalent}
  if $\gr(\tilde{\cal F}_1)\simeq \gr(\tilde{\cal F}_2)$.
 In notation,   $\tilde{\cal F}_1\stackrel{s}{\sim}\tilde{\cal F}_2$.
}\end{definition}

\bigskip

\bigskip
		
\subsection{A chamber structure on the space of  stability conditions}

We discuss in this subsection a chamber structure
 on the space of stability conditions used in this note.
As we remain in the case before Tom Bridgeland [Br],
 readers are referred to the classical work [Qin] of Zhenbo Qin
 for related early discussions and references.

 \bigskip

 \begin{flushleft}
 {\bf The space of stability conditions $\Stab^{1,[0]}(C\times Y)$.}
 \end{flushleft}
 Let
  $\DCG(C)$ be the degree class group of the nodal curve $C$,
  $\DCG^+(C)\subset \DCG(C)$ be the semigroup of effective classes,
  $\DCG(C)_{\Bbb R}:= \DCG(C)\otimes_{\Bbb Z}{\Bbb R}$, and
  $\DCG^+(C)_{\Bbb R}\subset \DCG(C)_{\Bbb R}$
    be the cone of effective classes spanned by ${\Bbb R}_{>0}$-rays
	through elements in $\DCG^+(C)\subset \DCG(C)_{\Bbb R}$,  and
  $\KCone(Y)$ be the K\"{a}hler cone of $Y$.	
Then, the data $(B+\sqrt{-1}J,L)$ that defines a central charge functional
  $Z^{B+\sqrt{-1}J,L}$  is parameterized by the set
  $$
    \Stab^{1,[0]}(C\times Y)\;
  	:=\;  H_2(Y;{\Bbb R})\times  \sqrt{-1}\cdot\KCone(Y)
	        \times \DCG^+(C)\,.
  $$
The natural topology on $\Stab^{1,[0]}(C\times Y)$
   as a subset in a vector space by definition
 coincides with the topology on $\Stab^{1,[0]}(C\times Y)$
  defined by treating its elements $(B+\sqrt{-1}J,L)$
  as $\hat{\Bbb H}_-$-valued functionals $Z^{B+\sqrt{-1}J,L}$
  on the set of $1$-dimensional coherent sheaves on $C\times Y$.
Denote also
 $$
  \Stab^{1,[0]}(C\times Y)_{\Bbb R}\;
   :=\;   H_2(Y;{\Bbb R})\times  \sqrt{-1}\cdot\KCone(Y)
	        \times \DCG^+(C)_{\Bbb R}\,,
 $$
 with multi-variable coordinates $(x_B+\sqrt{-1}\,x_J,x_L)$.

\bigskip

\begin{definition}
{\bf [space of stability conditions].} {\rm
 The set $\Stab^{1,[0]}(C\times Y)$ equipped with the above topology
  is called the {\it space of stability conditions}
  on the category of Fourier-Mukai transforms of dimension $1$ and width $[0]$
  from $C$ to $Y$.
}\end{definition}

\bigskip

\begin{notation}
{\bf [$\Stab^{1,[0]}(C\times Y)$ as abstract space]}  {\rm
 As there is no chance of confusion,
 for $Z=Z^{(B+\sqrt{-1}J,L)}$
   with $(B+\sqrt{-1}J,L) \in \Stab^{1,[0]}(C\times Y)$,
  we'll write also $Z\in \Stab^{1,[0]}(C\times Y)$
  with the latter treated as an abstract space of stability conditions.
}\end{notation}

\bigskip

\begin{flushleft}
{\bf Walls and chambers on $\Stab^{1,[0]}(C\times Y)$.}
\end{flushleft}
Since
  $\chi(\,\bullet\,)\in {\Bbb Z}$ and
  $\tilde{\beta}(\,\bullet\,)\in \NE(C\times Y)$
 are constant under flat deformations,
for the rest of the discussion in this subsection,
we shall fix a class
  $(\chi_0,\tilde{\beta}_0)\in {\Bbb Z}\times \NE(C\times Y)$   and
consider only $1$-dimensional coherent sheaves on $C\times Y$
  with Euler characteristic $\chi_0$ and curve class $\tilde{\beta}_0$.

Given
 $(B_1+\sqrt{-1}J_1,L_1)$, $(B_2+\sqrt{-1}J_2,L_2)
    \in\Stab^{1,[0]}(C\times Y)$,
let $Z_1 := Z^{B_1+\sqrt{-1}J_1,L_1}$ and
     $Z_2 := Z^{B_2+\sqrt{-1}J_2,L_2}$.	
Suppose that
  there exists a $1$-dimensional coherent sheaf $\tilde{\cal F}$
    with $\chi(\tilde{\cal F})=\chi_0$ and
           $\tilde{\beta}(\tilde{\cal F})=\tilde{\beta}_0$
  such that
   $\tilde{\cal F}$ is $Z_1$-stable but $Z_2$-unstable.
Then, there exists a saturated proper subsheaf
 $\tilde{\cal F}^{\prime}\subset \tilde{\cal F}$
 such that
 $$
   \mu^{Z_1}(\tilde{\cal F}^{\prime})\;
     <\; \mu^{Z_1}(\tilde{\cal F})
      \hspace{2em}\mbox{while}\hspace{2em}
   \mu^{Z_2}(\tilde{\cal F}^{\prime})\;
     >\; \mu^{Z_1}(\tilde{\cal F})\,.
 $$
In the explicit form, this says that
 $$
    \frac{
      \chi(\tilde{\cal F}^{\prime})
        - B_1\cdot\tilde{\beta}(\tilde{\cal F}^{\prime})}
	{(J_1+L_1)\cdot \tilde{\beta}(\tilde{\cal F}^{\prime})}\;
    < \;
   \frac{
      \chi_0
        - B_1\cdot\tilde{\beta}_0}
	{(J_1+L_1)\cdot \tilde{\beta}_0 }\; 	
   \hspace{1.2em}\mbox{while}\hspace{1.2em}
    \frac{
     \chi(\tilde{\cal F}^{\prime})
        - B_2\cdot\tilde{\beta}(\tilde{\cal F}^{\prime})}
	{(J_2+L_2)\cdot \tilde{\beta}(\tilde{\cal F}^{\prime})}\;
    > \;
   \frac{
      \chi_0
        - B_2\cdot\tilde{\beta}_0}
	{(J_2+L_2)\cdot \tilde{\beta}_0 }\; 	
 $$
Note that
  $0<\tilde{\beta}(\tilde{\cal  F}^{\prime})
      <\tilde{\beta}(\tilde{\cal F})=\tilde{\beta}_0$.
This motivates the following definition:

\bigskip

\begin{definition}
 {\bf [numerical-class-defined walls and chambers in $\Stab^{1,[0]}(C\times Y)$].}
{\rm
 For a fixed $(\chi_0,\tilde{\beta}_0)\in {\Bbb Z}\times \NE(C\times Y)$,
 let $(e,\tilde{\xi})\in {\Bbb Z}\times \NE(C\times Y)$
   with $0< \tilde{\xi}< \tilde{\beta}_0$.
 Let
   $$
      Q^{(\chi_0,\tilde{\beta}_0)}_{(e,\tilde{\xi})}
	  (x_B, x_J, x_L)\;
	  :=\;  \left( e - x_B\cdot\tilde{\xi}\right)
        	    \left((x_J, x_L)\cdot \tilde{\beta}_0 \right)\,
			  -\, \left(\chi_0 -  x_B\cdot\tilde{\beta}_0	\right)
			         \left((x_J, x_L)\cdot \tilde{\xi}\right)\,,
   $$
   a quadratic polynomial that is affine in the multivariable $x_B$
     and linear in the combined multivariable $(x_J,x_L)$.
 Define the {\it  numerical-class-defined wall}
  $W^{(\chi_0,\tilde{\beta}_0)}_{(e,\tilde{\xi})}$
   in $\Stab^{1,[0]}(C\times Y)_{\Bbb R}$ to be
  $$
	 W^{(\chi_0,\tilde{\beta}_0)}_{(e,\tilde{\xi})}\;
	   :=\;  \left\{   \left.
        	    \begin{array}{l}
			    (x_B+\sqrt{-1}\,x_J,x_L)\\[1.2ex]
		    	 \hspace{1em}
	    	 	 \in   \Stab^{1,[0]}(C\times Y)_{\Bbb R}
		        \end{array}	 \;
			                  \right|\;
			     Q^{(\chi_0,\tilde{\beta}_0)}_{(e,\tilde{\xi})}
	                  (x_B, x_J, x_L)\;
				 =\; 0
			   \right\}\,.			
	$$
  A connected component of
   $$
	  \Stab^{1,[0]}(C\times Y)_{\Bbb R} \;
	   -\,  \bigcup_{\mbox{\scriptsize
	           $\begin{array}{c}
	           (e,\tilde{\xi})\in {\Bbb Z}\times N\!E(C\times Y)\\[.6ex]
			    0\, <\, \tilde{\xi}\, <\, \tilde{\beta}_0
	            \end{array}$}}
	        W^{(\chi_0,\tilde{\beta}_0)}_{(e,\tilde{\xi})}
   $$
   is called a {\it numerical-class-defined chamber}
   of $\Stab^{1,[0]}(C\times Y)_{\Bbb R}$.
 For a given numerical-class-defined wall
   $W^{(\chi_0,\tilde{\beta}_0)}_{(e_0,\tilde{\xi}_0)}
      \subset \Stab^{1,[0]}(C\times Y)_{\Bbb R}$,
  the intersection	
   $$
	 W^{(\chi_0,\tilde{\beta}_0)}_{(e_0,\tilde{\xi}_0)}  \;
	   \bigcap\,  \bigcup_{\mbox{\scriptsize
	           $\begin{array}{c}
	            (e,\tilde{\xi})\in {\Bbb Z}\times N\!E(C\times Y)\\[.6ex]
			     0\, <\, \tilde{\xi}\, <\, \tilde{\beta}_0  \\[.6ex]
				 (e,\tilde{\xi})\; \ne\; (e_0,\tilde{\xi}_0)
	            \end{array}$}}
	        W^{(\chi_0,\tilde{\beta}_0)}_{(e,\tilde{\xi})}
   $$	
   induces a stratification of
   $W^{(\chi_0,\tilde{\beta}_0)}_{(e_0,\tilde{\xi}_0)}$
   by manifolds, which is called the {\it numerical-class-defined stratification}
   of the wall
   $W^{(\chi_0,\tilde{\beta}_0)}_{(e_0,\tilde{\xi}_0)}$.
}\end{definition}

\bigskip

For a fixed $(\chi_0,\tilde{\beta}_0)$,
 the above wall-and-chamber structure on
  $\Stab^{1,[0]}(C\times Y)_{\Bbb R}$ depends only on
    the structure of ${\Bbb Z}\times \NE(C\times Y)$  and
    the pairing
     $H^2(Y;{\Bbb C})\times NE(Y)\rightarrow {\Bbb C}$;
and, hence, the name.
However, it should be noted that
 as long as distinguishing stable conditions is concerned,
 a numerical-class-defined wall
  $W^{(\chi_0,\tilde{\beta}_0)}_{(e,\tilde{\xi})}$
  can truly separate two stability conditions whose associated subcategory
  of semistable objects are different
if, in addition, there exists a $1$-dimensional coherent sheaf $\tilde{\cal F}$
 on $C\times Y$ that contains a proper subsheaf $\tilde{\cal F}^{\prime}$
 such that $\chi(\tilde{\cal F}^{\prime})=e$ and
  $\tilde{\beta}(\tilde{\cal F}^{\prime})=\tilde{\xi}$.
I.e.\ the numerical equivalence class $(e,\tilde{\xi})$ is actually realized
 by some coherent subsheaf in our category.

\bigskip

\begin{definition}
 {\bf [actual walls and chambers in $\Stab^{1,[0]}(C\times Y)$].}
{\rm	
 A numerical-class-defined wall
  $W^{(\chi_0,\tilde{\beta}_0)}_{(e,\tilde{\xi})}$
  is called an {\it actual wall}  if $(e,\tilde{\xi})$	
  is realized by some coherent subsheaf of a coherent sheaf in our category.
 A connected component of
  $$
	\Stab^{1,[0]}(C\times Y)_{\Bbb R} \;
	   -\,  \bigcup_{\mbox{\scriptsize
	           $\begin{array}{c}
	           (e,\tilde{\xi})\in {\Bbb Z}\times N\!E(C\times Y)\\[.6ex]
			    0\, <\, \tilde{\xi}\, <\, \tilde{\beta}_0 \\[.6ex]
				\mbox{$(e,\tilde{\xi})$ is realized by a subobject}
	            \end{array}$}}
	        W^{(\chi_0,\tilde{\beta}_0)}_{(e,\tilde{\xi})}
  $$
  is called an {\it actual chamber}
  of $\Stab^{1,[0]}(C\times Y)_{\Bbb R}$.
  For a given actual wall
   $W^{(\chi_0,\tilde{\beta}_0)}_{(e_0,\tilde{\xi}_0)}
      \subset \Stab^{1,[0]}(C\times Y)_{\Bbb R}$,
  the intersection	
   $$
	 W^{(\chi_0,\tilde{\beta}_0)}_{(e_0,\tilde{\xi}_0)}  \;
	   \bigcap\,   \bigcup_{\mbox{\scriptsize
	           $\begin{array}{c}
	              (e,\tilde{\xi})\in {\Bbb Z}\times N\!E(C\times Y)\\[.6ex]
			       0\, <\, \tilde{\xi}\, <\, \tilde{\beta}_0                            \\[.6ex]
			 	   \mbox{$(e,\tilde{\xi})$ is realized by a subobject}    \\[.6ex]
				  (e,\tilde{\xi})\; \ne\; (e_0,\tilde{\xi}_0)
	            \end{array}$}}
	        W^{(\chi_0,\tilde{\beta}_0)}_{(e,\tilde{\xi})}
   $$	
   induces a stratification of
   $W^{(\chi_0,\tilde{\beta}_0)}_{(e_0,\tilde{\xi}_0)}$
   by manifolds, which is called the {\it actual stratification} of the wall
   $W^{(\chi_0,\tilde{\beta}_0)}_{(e_0,\tilde{\xi}_0)}$.
 }\end{definition}	
	
\bigskip

\noindent
By construction,
the numerical-class-defined wall-and-chamber structure
 on $\Stab^{1,[0]}(C\times Y)_{\Bbb R}$
 is a refinement of the actual wall-and-chamber structure
 on $\Stab^{1,[0]}(C\times Y)_{\Bbb R}$.
The former depends only on the intersection theory
   on $C\times Y$
 while the latter is much harder to decide.

Two immediate consequences follow:

\bigskip

\begin{lemma}
 {\bf [coincidence of semistability and stability].}
 For $(B+\sqrt{-1}J,L)$ in an actual chamber of
   $\Stab^{1,[0]}(C\times Y)_{\Bbb R}$,
  $Z$-semistable implies $Z$-stable and, hence,
  the notion of $Z$-semistability and of $Z$-stability coincide
   for objects in our category.
\end{lemma}

\bigskip

\begin{lemma}
 {\bf [equivalent central charges].}
 If $Z_1$ and $Z_2$ lie in the same actual chamber
   of $\Stab^{1,[0]}(C\times Y)_{\Bbb R}$,
 then
  an object $\tilde{\cal F}$ in our category
      is $Z_1$-(semi)stable
  if and only if it is $Z_2$-(semi)stable.
\end{lemma}

\bigskip

\begin{example}
{\bf [chamber structure on $\Stab^{1,[0]}({\Bbb P}^1\times Y)_{\Bbb R}$,
           $Y$ quintic Calabi-Yau $3$-fold].}  {\rm
 Let $Y$ be a quintic Calabi-Yau $3$-fold in ${\Bbb P}^4$.
 Then,
  $\dim H_2(Y,{\Bbb R})=  h^{1,1}=1$;
  $\DCG^+({\Bbb P}^1)={\Bbb Z}_{>0}$ and, hence;
  $\Stab^{1,[0]}({\Bbb P}^1\times Y)_{\Bbb R}
      \simeq {\Bbb R}\times {\Bbb R}_{>0}\times {\Bbb R}_{>0}$,
	whose coordinates will be denoted by $(x_B, x_J, x_L)$;  and
  $\NE({\Bbb P}^1\times Y)= {\Bbb Z}_{>0}\oplus  {\Bbb Z}_{>0}$.	
 Let $(e,\tilde{\xi})=(e; m,n)
           \in {\Bbb Z}\times ({\Bbb Z}_{>0}\oplus {\Bbb Z}_{>0})$.
 For a fixed $(\chi_0,\tilde{\beta}_0)= (\chi_0; m_0,n_0)$,
the wall $W^{(\chi_0; m_0,n_0)}_{(e; m,n)}$,
  with $0<m<m_0$ and $0<n<n_0$, is given by zero-locus
  of the function in $(x_B,x_J,x_L)$
 $$
   \begin{array}{l}
    (e-nx_B)(m_0x_L+n_0x_J)
	   - (\chi_0-n_0x_B)(mx_L+nx_J)\\[1.6ex]
	  \hspace{1em}
      =\;  \begin{vmatrix}e & m \\ \chi_0 & m_0 \end{vmatrix} x_L\,
			+\, \begin{vmatrix}e & n \\ \chi_0 & n_0 \end{vmatrix} x_J\,
			+\,\begin{vmatrix}m & n \\ m_0 & n_0 \end{vmatrix} x_Bx_L\,,
   \end{array}
 $$
  which is
    a truncated hyperbolic paraboloid (i.e.\ saddle-shaped quadric surface)
      in ${\Bbb R}\times {\Bbb R}_{>0}\times {\Bbb R}_{>0}$
     if $\begin{vmatrix}e & n \\ \chi_0 & n_0 \end{vmatrix}\ne 0$,
	 and
   the union of two truncated hyperplanes, described by the equation
    $
	   x_L
	    \Big(\begin{vmatrix}e & m \\ \chi_0 & m_0 \end{vmatrix}\,
	     +\,\begin{vmatrix}m & n \\ m_0 & n_0 \end{vmatrix} x_B
		\Big) =0$
     if  $\begin{vmatrix}e & n \\ \chi_0 & n_0 \end{vmatrix}=0$.
  %
  %
 The latter can occur only for finitely many $(e,n)$'s.
 In general, when $e\rightarrow\pm$, the hyperbolic paraboloid
  tends to the hyperplane $n_0x_J+m_0x_L=0$, which has no intersection
  with $\Stab^{1,[0]}(C\times Y)_{\Bbb R}$.
 Thus we have a locally finite system of walls that defines
  the numerical-class-defined chamber structure of
  $\Stab^{1,[0]}(C\times Y)_{\Bbb R}$.
 This turns out to be a general feature.
}\end{example}

\bigskip

\bigskip

\begin{flushleft}
{\bf Local finiteness of walls.}
\end{flushleft}
\begin{proposition}
 {\bf [local finiteness of walls].}
  For any $p\in Stab^{1,[0]}(C\times Y)_{\Bbb R}$,
  there exists an open neighborhood $U$ of $p$
     in $\Stab^{1,[0]}(C\times Y)_{\Bbb R}$
   such that there are only finitely many $(e,\tilde{\xi})$
    with $W^{(\chi_0,\tilde{\beta}_0)}_{(e,\tilde{\xi})}
	          \cap U \ne \emptyset$.
\end{proposition}

\begin{proof}
 Recall the defining equation for a wall:
  $$
    Q^{(\chi_0,\tilde{\beta}_0)}_{(e,\tilde{\xi})}
	  (x_B, x_J, x_L)\;
	  :=\;  \left( e - x_B\cdot\tilde{\xi}\right)
        	    \left((x_J, x_L)\cdot \tilde{\beta}_0 \right)\,
			  -\, \left(\chi_0 -  x_B\cdot\tilde{\beta}_0	\right)
			         \left((x_J, x_L)\cdot \tilde{\xi}\right)\;
	  =\; 0\,.
  $$
 Since
   $\chi_0$ and $\tilde{\beta}_0$ are fixed and
   there are only finitely many choices of $\tilde{\xi}$,
 we only need to focus on the behavior of equation
  when $e\rightarrow \pm\infty$.
 In which case, the equations are approximating
  the fixed equation $(x_J,x_L)\cdot\tilde{\beta}_0=0$.
 But $(x_J,x_L)\cdot\tilde{\beta}_0=0$
  has no solution in
  $\Stab^{1,[0]}(C\times Y)_{\Bbb R}$.
 This proves the proposition.

\end{proof}

\bigskip

\bigskip

\begin{flushleft}
{\bf Behavior of the moduli stack of stable objects when crossing an actual wall:\\
         a conjecture. }
\end{flushleft}
Let $\Delta^-$ and $\Delta^+$ be two actual chambers in
 $\Stab^{1,[0]}(C\times Y))_{\Bbb R}$
 such that
  both $\Delta^-\cap \Stab^{1,[0]}(C\times Y)$ and
     $\Delta^+\cap \Stab^{1,[0]}(C\times Y)$ are non-empty
  and that the intersection
    $\overline{\Delta^-}\cap \overline{\Delta^+}$
	of their closure contains a maximal stratum of a wall
    $W^{(\chi_0,\tilde{\beta}_0)}_{(e_0,\tilde{\xi}_0)}$.
Then,
 $W^{(\chi_0,\tilde{\beta}_0)}_{(e_0,\tilde{\xi}_0)}$
 is the only wall that intersects the interior
  $\Int (\overline{\Delta^-}\cup \overline{\Delta^+})$	
  of $\overline{\Delta^-}\cup \overline{\Delta^+}$.
Recall the defining equation
 $Q^{(\chi_0,\tilde{\beta}_0)}_{(e_0,\tilde{\xi}_0)}$
 of $W^{(\chi_0,\tilde{\beta}_0)}_{(e_0,\tilde{\xi}_0)}$.
Let $p^-\in \Delta^-$, $p^+\in\Delta^+$,  $p^0 \in$
 a maximal stratum in
 $W^{(\chi_0,\tilde{\beta}_0)}_{(e_0,\tilde{\xi}_0)}$
 that lies in $\overline{\Delta^-}\cap \overline{\Delta^+}$.
Let $Z^-$, $Z^+$, $Z^0$ be their associate central charge functional
 respectively.
Up to a relabelling, we may assume that
 $Q^{(\chi_0,\tilde{\beta}_0)}_{(e_0,\tilde{\xi}_0)}
     (p^-)<0$   and
 $Q^{(\chi_0,\tilde{\beta}_0)}_{(e_0,\tilde{\xi}_0)}
       (p^+)>0$.
Let ${\cal S}^-$ (resp.\ ${\cal S}^0$, ${\cal S}^+$)
 be the moduli stack of $Z^-$-stable
 (resp.\  $Z^-$-semistable, $Z^+$-stable)
 $1$-dimensional coherent sheaves on $C\times Y$
 with Euler characteristic $\chi_0$ and curve class $\tilde{\beta}_0$.
	
Given any nonzero proper inclusion pair
 $\tilde{\cal F}^{\prime}\subset \tilde{\cal F }$
 with $\chi(\tilde{\cal F})=\chi_0$ and
   $\tilde{\beta}(\tilde{\cal F})=\tilde{\beta}_0$,
 there are four situations:
 \begin{itemize}
  \item[(1)]
   $\mu^{Z^-}(\tilde{\cal F}^{\prime})- \mu^{Z^-}(\tilde{\cal F})<0\;$ and
   $\;\mu^{Z^+}(\tilde{\cal F}^{\prime})- \mu^{Z^+}(\tilde{\cal F})<0$.

  \item[(2)]
   $\mu^{Z^-}(\tilde{\cal F}^{\prime})- \mu^{Z^-}(\tilde{\cal F})>0\;$ and
   $\;\mu^{Z^+}(\tilde{\cal F}^{\prime})- \mu^{Z^+}(\tilde{\cal F})<0$.\

  \item[(3)]
   $\mu^{Z^-}(\tilde{\cal F}^{\prime})- \mu^{Z^-}(\tilde{\cal F})<0\;$ and
   $\;\mu^{Z^+}(\tilde{\cal F}^{\prime})- \mu^{Z^+}(\tilde{\cal F})>0$.

  \item[(4)]
   $\mu^{Z^-}(\tilde{\cal F}^{\prime})- \mu^{Z^-}(\tilde{\cal F})>0\;$ and
   $\;\mu^{Z^+}(\tilde{\cal F}^{\prime})- \mu^{Z^+}(\tilde{\cal F})>0$.
 \end{itemize}
When one deforms both the stability conditions $Z^-$, $Z^+$ to $Z^0$,
 each strict inequality becomes at worse inequality of the same direction
 (i.e.\ $<$ becomes $<$ or $\le$
        and $>$ becomes $>$ or $\ge$ ).
It follows that one has tautological morphisms of stacks:
 $$
   \xymatrix{
     {\cal S}^-\ar[rrd]  &&&& {\cal S}^+\ar[lld] \\
       && {\cal S}^0          &&
   }
 $$
 by sending a family of $Z^-$-stabe (resp.\ $Z^+$-stable) objects
  to the same family of objects, which now become a family
  of $Z^0$-semistable objects.
In view that in good cases our non-Geometric-Invariant-Theory situation
  is a deformation from a Geometric-Invariant-Theory situation
 and together with the compactness of
 ${\cal S}^-$, ${\cal S}^0$, and ${\cal S}^+$,
 to be proved in the next section (Theorem~3.1),
 one has the following conjecture:

\bigskip

\begin{conjecture}
 {\bf [wall-crossing of moduli stacks].}
 There exists projective schemes $S^-$, $S^0$, and $S^+$,
     with built-in morphisms
      ${\cal S}^-\rightarrow S^-$,
      ${\cal S}^0\rightarrow S^0$, ${\cal S}^+\rightarrow S^+$
     that parameterize $S$-equivalence classes of semistable objects
     in the related moduli stack,
  and birational surjections $S^-\longrightaarrow S^0 \longleftaarrow S^+$
  such that the following diagram commute
  $$
   \xymatrix{
    & {\cal S}^- \ar[rrd] \ar[dd]   &&&& {\cal S}^+ \ar[lld] \ar[dd]    \\
    &   && {\cal S}^0 \ar[dd]                                                                     \\
    &   S^-\ar@{->>}[rrd]            &&&&   S^+\ar@{->>}[lld]             \\
    &   &&   S^0     &&&.
    }
 $$
\end{conjecture}

\bigskip

\begin{remark}
{\it $[$stringy correction/deformation to walls/chambers$]$.} {\rm
 It should be noted that
 when the stringy correction $O(\alpha^{\prime})$
   to our central charge functional $Z$ is taken into account,
 the walls and, hence, the chamber structure on $\Stab^{1,[0]}(C\times Y)_{\Bbb R}$
   described in this subsection are subject to a corresponding
   stringy correction/deformation as well.
}\end{remark}

\bigskip

\bigskip

\section{Compactness of the moduli stack
                  $\FM^{1,[0];\smallZss}_{C_{\cal M}/{\cal M}}(Y;c)$
                  of $Z$-semistable Fourier-Mukai transforms}

We now prove the main theorem of this note:

\bigskip

\begin{stheorem}
 {\bf [$\FM^{1,[0];\scriptsizeZss}_{C_{\cal M}/{\cal M}}(Y;c)$ compact].}
   Let
    $(Y, B+\sqrt{-1}J)$ be a projective Calabi-Yau $3$-fold
       with a fixed complexified K\"{a}hler class,
    ${\cal M}$  be a compact stack of nodal curves,
    $C_{\cal M}/{\cal M}$ be the associated universal curve over ${\cal M}$
	   with a fixed relative polarization class $L$.
   Then
     the moduli stack   $\FM^{1,[0];\scriptsizeZss}_{C_{\cal M}/{\cal M}}(Y;c)$
      of $Z^{B+\sqrt{-1}J,L}$-semi-stable Fourier-Mukai transforms
	     of dimension $1$,  width $[0]$, and central charge $c\in \hat{\Bbb H}_-$
     from fibers of $C_{\cal M}/{\cal M}$  to $Y$
	 is compact.
\end{stheorem}

\bigskip

\bigskip

\begin{flushleft}
{\bf Boundedness of
         $\FM^{1,[0];\scriptsizeZss}_{C_{\cal M}/{\cal M}}(Y;c)$.}
\end{flushleft}
Recall first the Kleiman's Boundedness Criterion:

\bigskip

\begin{stheorem}
{\bf [Kleiman's boundedness criterion].}{\rm  (Cf.\ [H-L: Theorem~1.7.8], [Kl].)}
 Let ${F_s}$ be a family of coherent sheaves on $X$ with the same Hilbert polynomial $P$.
 Then this family is bounded if and only if there are constants $c_i$, $i=0,\,\ldots\,,\, d=\degree(P)$
  such that for every $F_s$ there exists an $F_s$-regular sequence of hyperplane sections
   $H_1,\,\cdots\,,\, H_d$ such that $h^0(F_s|_{\cap_{j\le i} H_j})\le c_i$,
   for all $i=0,\,\ldots\,,\, d$.
\end{stheorem}

\bigskip

Let $H$ be a relative ample class of $(C_{\cal M}\times Y)/{\cal M}$.
Then, it follows from Lemma~2.1.7
   that
   $$
      \FM^{1,[0];\scriptsizeZss}_{C_{\cal M}/{\cal M}}(Y;c)\;
	  =\; \coprod_{P^H_i\in \scriptsizePoly^H(Z^{B+\sqrt{-1}J,L}=c)}\,
	         \FM^{1,[0];\scriptsizeZss}_{C_{\cal M}/{\cal M}}(Y;c)^{P^H_i}
  $$
    is a finite disjoint union of substacks
	$\FM^{1,[0];\scriptsizeZss}_{C_{\cal M}/{\cal M}}(Y;c)^{P^H_{i}}$,
     with the latter the union of connected components of
	 $\FM^{1,[0];\scriptsizeZss}_{C_{\cal M}/{\cal M}}(Y;c)$
	  whose elements have the same Hilbert polynomial $P^H_{i}$.
Recall the following estimate:

\bigskip

\begin{sproposition}
 {\bf [bound on $h^0(\tilde{\cal F })$ via $P$-slope].}
 {\rm ([H-L: Theorem~3.3.1], [LP], [Ma2], [Si].) }
 Let $\tilde{\cal F}$ be a purely $1$-dimensional coherent sheaf
  on $C\times Y$ for $[C]\in {\cal M}$ .
 Then,
  $$
    h^0(\tilde{\cal F})\;
	 \le\;
	          \left(H\cdot \tilde{\beta}(\tilde{\cal F})\right)\,
	          \left[
    		    \mu^P_{max}(\tilde{\cal F})\,
			    +\,  \frac{1}{2}
				         \left( H\cdot\tilde{\beta}(\tilde{\cal F})\right)
						 \left( H\cdot\tilde{\beta}(\tilde{\cal F})\,+\,1\right)\,
				-1		
	          \right]_+\,,
  $$
  where $[\,\bullet\,]:= \max\{0, \,\bullet\,\}$.
\end{sproposition}

\bigskip

\noindent
Combined with Lemma~2.1.8,
one has then the estimate:

\bigskip

\begin{sproposition}
{\bf [bound on $h^0(\tilde{\cal F })$ via $Z$-slope].}
 There exist constants $a_1>0$ and $a_0\in {\Bbb R}$ that depend only on
  $(Y,B+\sqrt{-1}J)$, $(C_{\cal M}/{\cal M}, L)$, and $H$
  such that
  for all
   $[\tilde{\cal F}] \in
      \FM^{1,[0];\scriptsizeZss}_{C_{\cal M}/{\cal M}}(Y;c)$,
  $$
    h^0(\tilde{\cal F})\;
	 \le\;
	          \left(H\cdot \tilde{\beta}(\tilde{\cal F})\right)\,
	          \left[
    		    \left(a_1\,\frac{\Real c}{-\Imaginary c}\,+\,a_0\right)\,
			    +\,  \frac{1}{2}
				         \left( H\cdot\tilde{\beta}(\tilde{\cal F})\right)
						 \left( H\cdot\tilde{\beta}(\tilde{\cal F})\,+\,1\right)\,
				-1		
	          \right]_+\,,
  $$
  where  $[\,\bullet\,]:= \max\{0, \,\bullet\,\}$.
\end{sproposition}

On the other hand,
 $h^0(\tilde{\cal F}|_H)=H\cdot\tilde{\beta}(\tilde{\cal F})$
  is the leading coefficient of the Hilbert polynomial $P^H(\tilde{\cal F})$
   of $\tilde{\cal F}$   and, hence,
  is constant for all
    $[\tilde{\cal F}]  \in
	    \FM^{1,[0];\scriptsizeZss}_{C_{\cal M}/{\cal M}}(Y;c) ^{P^H_i}$.

It follows now from the Kleiman's Boundedness Criterion
that each $\FM^{1,[0];\scriptsizeZss}_{C_{\cal M}/{\cal M}}(Y;c)^{P^H_i}$
            is bounded and, hence,
that $\FM^{1,[0];\scriptsizeZss}_{C_{\cal M}/{\cal M}}(Y;c)$ is bounded.

\bigskip

\bigskip

\begin{flushleft}
{\bf Completeness of $\FM^{1,[0];\scriptsizeZss}_{C_{\cal M}/{\cal M}}(Y;c)$.}
\end{flushleft}
{From} the basic properties
  of twisted central charges and
  of the associated stability conditions and Harder-Narasimhan filtrations
 given in Sec.~2,
the completeness of the stack
$\FM^{1,[0];\scriptsizeZss}_{C_{\cal M}/{\cal M}}(Y;c)$
follows from the Langton's argument [La]
 through elementary modifications,
 with (e.g., for the modified and generalized presentation in [H-L: Sec.~2.B])
  Hilbert polynomials and reduced Hilbert polynomials
  replaced by twisted central charges and associated slope respectively
 and using the properness of the related relative $\Quot$-schemes.
We convert the proof of
   Stacy Langton [La] and  Daniel Huybrechts and Manfred Lehn [H-L: Theorem~2.B.1]
  to our case below for the thoroughness of the discussion.
Readers are referred to [La] and [H-L] for the original ideas and proofs.
(Note that the inductive proof of Huybrechts and Lehn via quotient category
    of coherent sheaves becomes superficial
   when one concerns only with $1$-dimensional coherent sheaves.)

\bigskip

Let $R$ be a discrete valuation ring
  with maximal ideal $m=(t)$, residue field $R/(t)\simeq k\simeq {\Bbb C}$
         and quotient field $K$.
When needed, a scheme over $\Spec R$
   (resp.\ the generic point  $\Spec K$ and the closed point $\Spec k\, \in \, \Spec R $)
  will be denoted by $X_R$ (resp.\ $X_K$ and $X_k$ );  and
similarly for a coherent sheaf.

\bigskip

\begin{sproposition}
 {\bf [valuative criterion].}
  Let
   $(C_K, L_K)$ be a nodal curve with a polarization class over $\Spec K$
     from a morphism $\Spec K\rightarrow {\cal M}$   and
  $\tilde{\cal F}_K$ be $Z$-semistable coherent sheaf on $(C_K\times Y)/K$.
 Then, up to a field extension $K\subset K^{\prime}$ of finite degree,
   $\tilde{\cal F}_K$ extend to a coherent sheaf
   $\tilde{\cal F }_{R^{\prime}}$  on $(C_{R^{\prime}\times Y})/R^{\prime}$
   that is flat over $R^{\prime}$ and has $\tilde{\cal F}_k$ $Z$-semistable.
 Here, $R^{\prime}\leftarrow R$
   is the discrete valuation ring with residue firld $k$ and quotient field $K^{\prime}$.
\end{sproposition}

\bigskip

We proceed to prove the proposition.
Since ${\cal M}$ is compact, up to a field extension (still denoted by $K$
 and the associated discrete valuation ring by $R$ for simplicity of notations),
 $(C_K,L_K)$ extends to $(C_R,L_R)$
  from a morphism $\Spec R\rightarrow {\cal M}$
  that extends the given (or induced under a field extension)
  $\Spec K\rightarrow {\cal M}$.
(Note that
 there is an embedding $C_R\hookrightarrow {\Bbb P}^N_R$  over $R$,
    for some $N>0$.
 Thus, one can treat our problem exchangeably as a problem of coherent sheaves
  on the fixed smooth ${\Bbb P}^N_R\times Y$
  with their support contained in the closed subscheme $C_R\times Y$
  if one prefers.)
One can extend $\tilde{\cal F}_K$ to a coherent sheaf
   $\tilde{\cal F}_R =: \tilde{\cal F}_R^1$
   of coherent sheaves on $(C_R\times Y)/R$ that is flat over $R$.
(Cf. \ [Hart: II, Exercise 5.15]
         for an extension inside $j_{\ast}\tilde{\cal  F}_K$,
		  which must be flat over $R$.
		 Here, $j:C_K\times Y \rightarrow C_R\times Y$
   		  is the built-in open immersion.)	
			
If $\tilde{\cal F}_k^1$	is $Z$-semistable, then we are done. 		
Otherwise,
  let $\tilde{\cal B}^1$ be the maximal destabilizing subsheaf of $\tilde{\cal F}_k^1$,
   $\tilde{\cal G}^1:= \tilde{\cal F}_k^1 / \tilde{\cal B}^1$,   and
   $$
    \tilde{\cal F}_R^2\;  := \;
	 \Ker(\tilde{\cal F}_R^1\rightarrow \tilde{\cal G}^1)
   $$
   be the {\it elementary modification} of $\tilde{\cal F}_R^1$ at $\tilde{\cal G}^1$.
Here,  $\tilde{\cal F}_R^1\rightarrow \tilde{\cal G}^1$ is the composition
  $\tilde{\cal F}_R^1\rightarrow \tilde{\cal F}_k^1\rightarrow \tilde{\cal G}^1$
  of quotient-sheaf homomorphisms.
If $\tilde{\cal  F}_k^2$ is semistable, then we are done.
Otherwise, one can iterate the above elementary modification
  now on $\tilde{\cal F}_R^2$
   at $\tilde{\cal G}^2:=  \tilde{\cal F}^2_k / \tilde{\cal B}^2$,
    where $\tilde{B}^2$ is the maximal destabilizing subsheaf of $\tilde{\cal F}_k^2$.
If $\tilde{\cal B}^i=\tilde{\cal  F}_k^i$ for some $i\ge 1$, then we are done.

{\it Assume that this is not true},
 we obtain then a sequence of proper subsheaves that are flat over $R$:
  $$
    \cdots\;  \subsetneqq \; \tilde{\cal F}_R^{\, i+1} \;
	  \subsetneqq \; \tilde{\cal F}_R^{\,i} \;
	  \subsetneqq \; \cdots  \;
	   \subsetneqq \; \tilde{\cal F}_R^1 \,.
  $$
  Note that
   $t\tilde{\cal F}_R^{\,i-1}
      \subset \tilde{\cal F}_R^{\,i}\subset \tilde{\cal F}_R^{\, i-1}$ and, hence,
    $t^{i-1}\tilde{\cal F}_R^1\subset \tilde{\cal F}_R^{\, i}\subset \tilde{\cal F}_R^1$.
By construction, one has the built-in exact sequence
    $$
      0\;  \longrightarrow\;
	   \tilde{\cal B}^i\;  \longrightarrow\;
	   \tilde{\cal F}^i_k\; \longrightarrow\; \tilde{\cal  G}^i\;
	   \longrightarrow \; 0\,,
   $$
  for all $i$.

\bigskip

\begin{slemma}
 {\bf [short exact sequence].}
  There is also an exact sequence
   $$
      0\;  \longrightarrow\;
	   \tilde{\cal G}^{i-1}\;  \longrightarrow\;
	   \tilde{\cal F}^i_k\; \longrightarrow\; \tilde{\cal B}^{i-1}\;
	   \longrightarrow \; 0\,,
   $$
   for all $i$.
\end{slemma}

\begin{proof}
 The short exact sequence
   $$
     0\;\longrightarrow\;
	   \tilde{\cal F}_R^i\;\longrightarrow\;
	   \tilde{\cal F}_R^{i-1}\; \longrightarrow\;
	   \tilde{\cal G}^{i-1}\; \longrightarrow\; 0
   $$
  induces a long exact sequence of ${\cal O}_{C_R\times Y}$-modules
   $$
   \begin{array}{l}
      \cdots\;\longrightarrow\;
	   \Tor^1_{C_R\times Y}({\cal O}_{C_k\times Y}, \tilde{\cal  F}^{i-1}_R)\;
	    \longrightarrow\;
	   \Tor^1_{C_R\times Y}({\cal O}_{(C_k\times Y)}, \tilde{\cal G}^{i-1})
	                      \\[1.6ex]
	 \hspace{20em}\longrightarrow\;
	  \tilde{\cal F}^i_k\;\longrightarrow\;
      \tilde{\cal F}^{i-1}_k\; \longrightarrow\; 	
	  \tilde{\cal G}^{i-1}\;\longrightarrow\; 0\,.
   \end{array}	
  $$
 The lemma follows from the observations that
  $$
   \begin{array}{c}
     \Tor^1_{C_R\times Y}({\cal O}_{C_k\times Y},  \tilde{\cal  F}^{i-1}_R  )\;
       =0\; , \\[1.6ex]
     \Tor^1_{C_R\times Y}({\cal O}_{(C_k\times Y)}, \tilde{\cal G}^{i-1})\;
       \simeq\; (t)\otimes_{{ C_R\times Y}}\tilde{\cal G}^{i-1}\;
	   \simeq\; \tilde{\cal G}^{i-1}\;,  \\[1.6ex]
	 \Image(\tilde{\cal F}^i_k\rightarrow \tilde{\cal F}^{i-1}_k)\;
   	   =\; \tilde{\cal B}^{i-1}\;.
   \end{array}
  $$

\end{proof}

\bigskip

These two sets of exact sequences give rise to two sequences of homomorphisms of
 coherent sheaves on $C_k\times Y$:
    $$
	  \cdots\; \longrightarrow\;  \tilde{\cal  B}^{i+1}\; \longrightarrow\;
	   \tilde{\cal B}^i\; \longrightarrow\; \cdots\; \longrightarrow\; \tilde{\cal B}^1
	$$
    and 	
    $$
	  \tilde{\cal G}^1\; \longrightarrow\; \cdots\; \longrightarrow\;
	   \tilde{\cal  G}^i\; \longrightarrow\;  \tilde{\cal G}^{i+1}\; \longrightarrow\; \cdots\,.
	$$   	

\bigskip

\begin{slemma}
 {\bf [stationary behavior of
           $\{\tilde{\cal B}^i\}_i$,
		   $\{\tilde{\cal G}^i\}_i$, and $\{\tilde{\cal F}^i_k\}_i$].}
 There exists an $i_0\ge 1$ such that for all $i\ge i_0$,
   the homomorphisms
      $\tilde{\cal B}^{i+1}\rightarrow  \tilde{\cal B}^i$  and
  	  $\tilde{\cal G}^i \rightarrow \tilde{\cal G}^{i+1}$
    are isomorphisms   and
   $\tilde{\cal F}_k^i\simeq \tilde{\cal B}\oplus \tilde{\cal G}$,
 where
  $\tilde{\cal B}\simeq  \tilde{\cal B}^i$  and
  $\tilde{\cal G}\simeq \tilde{\cal G}^i$ for any $i\ge i_0$.
\end{slemma}

\begin{proof}
 Observe first that
  $$
     \{\tilde{\beta}(\tilde{\cal F}^{\prime})\,|\,
	       \mbox{$\tilde{\cal F}^{\prime}$ is a subsheaf of $\tilde{\cal F}^i_k$
                     		   for some $i$, $i \ge 1$} \}
  $$
  is a finite subset of $N_1(C_k\times Y)_{\Bbb Z}$.
 Thus,  for any subsheaf of $\tilde{\cal F}^i_k$, $i\ge 1$,
    $\mu^Z(\tilde{\cal F}^{\prime})$ takes values
     in a discrete affine lattice of rank $1$
	(i.e.\ a shift -- due to the effect of $B$ in $B+\sqrt{-1}J$ --
	    of an additive subgroup isomorphic to ${\Bbb Z}$)
	in ${\Bbb R}$ that is independent of $i$.
 Back to our problem,
 since $\tilde{\cal B}^i$ are $Z$-semistable,
    $$
	  \cdots\; \le\; \mu^Z( \tilde{\cal  B}^{i+1})\; \le\;
	   \mu^Z(\tilde{\cal B}^i)\; \le\; \cdots\; \le\; \mu^Z(\tilde{\cal B}^1)\,.
	$$
 Since
    $\mu^Z(\tilde{\cal B}^i)
	     > \mu^Z(\tilde{\cal F }^i_k)
		 =    \mu^Z(\tilde{\cal F}_K)$ for all $i$,
  the above observation implies that
   there exists an $i_0\ge 1$
    such that
      $$
 	   \cdots\; =\; \mu^Z( \tilde{\cal  B}^{i+1})\; =\;
	    \mu^Z(\tilde{\cal B}^i)\; =\; \cdots\; =\; \mu^Z(\tilde{\cal B}^{i_0})
	  $$
      for all $i\ge i_0$.
 On the other hand, note that
  $\tilde{\cal B}^{i+1}/(\tilde{\cal  B}^{i+1}\cap \tilde{\cal G}^i)$
  is isomorphic to a non-zero subsheaf of $\tilde{\cal B}^i$ and,
  since $\tilde{\cal B}^i$'s are $Z$-semistable,

  $$
   \mu^Z(\tilde{\cal B}^{i+1})\;
    \le\; \mu^Z(\tilde{\cal B}^{i+1}/(\tilde{\cal B}^{i+1}\cap \tilde{\cal G}^i))\;
    \le\; \mu^Z(\tilde{\cal B}^i)	
  $$
  with the equalities hold if and only if $\tilde{\cal B}^{i+1}\cap \tilde{\cal G}^i=0$.
 It follows that $\tilde{\cal B}^{i+1}\cap \tilde {\cal G}^i=0$  for $i\ge i_0$,
  which implies that
   the homomorphisms,
    $\tilde{\cal B}^{i+1}\rightarrow \tilde{\cal B}^i$ and
    $\tilde{\cal G}^i\rightarrow \tilde{\cal  G}^{i+1}$,
   are injective for $i\ge i_0$.
 Being non-zero subsheaves,
   $\tilde{\beta}(\tilde{\cal B}^{i+1})$
    is an effective sub-$1$-cycle of $\tilde{\beta}(\tilde{\cal B}^i)$   and
  $\tilde{\beta}(\tilde{\cal G}^i)$
    is an effective sub-$1$-cycle of $\tilde{\beta}(\tilde{\cal G}^{i+1})$.
 Since all are effective sub-$1$-cycles of $\tilde{\beta}(\tilde{\cal F}_k)$,
  there exists an $i_1\ge i_0$  such that
   $$
    \tilde{\beta}(\tilde{\cal B}^{i+1})\;
 	  =\; \tilde{\beta}(\tilde{\cal B}^i)
     \hspace{2em}\mbox{and}\hspace{2em}
    \tilde{\beta}(\tilde{\cal G}^i)\;
       =\; \tilde{\beta}(\tilde{\cal G}^{i+1})\,,
   $$
   for all $i\ge i_1$.
 It follows that
  $$
    Z^{B+\sqrt{-1}J,L}(\tilde{\cal B}^{i+1})\;
      =\;  Z^{B+\sqrt{-1}J,L}(\tilde{\cal B}^i)
    	 \hspace{1em}\mbox{and, hence,}\hspace{1em}
    Z^{B+\sqrt{-1}J,L}(\tilde{\cal G}^i)\;
      =\; Z^{B+\sqrt{-1}J,L}(\tilde{\cal G}^{i+1})\,,
  $$	
  for all $i\ge i_1$, and that
  the inclusions of the purely $1$-dimensional coherent sheaves,
    $\tilde{\cal B}^{i+1}\rightarrow \tilde{\cal B}^i$ and
    $\tilde{\cal G}^i\rightarrow \tilde{\cal  G}^{i+1}$,
   are isomorphisms in dimension-$1$, for $i\ge i_1$.
 Now
  $$
    \tilde{\cal G}^{i_1}\;\subset \; \cdots\;  \subset\;
    \tilde{\cal G}^i\;  \subset\; \tilde{\cal G}^{i+1}\;  \subset\; \cdots
  $$
   is an inclusion sequence of purely $1$-dimensional coherent sheaves
     which are isomorphic in dimension $1$.
 In particular, their reflexive hulls
  $(\tilde{\cal G }^i)^{DD}
      := \Extsheaf^3_{C_k\times Y}
            ( \Extsheaf^3_{C_k\times Y}
			                  (\tilde{\cal G}^i , \omega_{C_k\times Y}        )\,,\,
				 \omega_{C_k\times Y})$	
  are canonically isomorphic, induced by the inclusions.
 Here, $\omega_{C_k\times Y}$ is the dualizing sheaf of $C_k\times Y$
  and the superscript $3$ is the codimension of $\tilde{\cal G}^i$ in $C_k\times Y$.
 It follows that
  there exists an $i_2\ge i_1$ such that
   $\tilde{\cal G}^i\stackrel{\sim}{\rightarrow}{\cal G}^{i+1}$
   for all $i\ge i_2$.
 Recall the exact sequence
  $0\rightarrow \tilde{\cal G}^{i-1}\rightarrow \tilde{\cal F}^i_k
      \rightarrow \tilde{\cal B}^{i-1}\rightarrow 0$;
  one has thus in addition that
   the exact sequence
     $0\rightarrow  \tilde{\cal B}^i  \rightarrow   \tilde{\cal F}^i_k
	     \rightarrow  \tilde{\cal G}^i  \rightarrow 0$
     splits and
   $\tilde{\cal  B}^{i+1}\stackrel{\sim}{\rightarrow}  \tilde{\cal B}^i$,
   for all $i\ge i_2+1$.	

 With $i_2+1$ re-denoted by $i_0$, this proves the lemma.

\end{proof}

\bigskip

\noindent
After removing finitely many terms and relabelling,
  we may assume without loss of generality that $i_0=1$ in the following discussion.

Define now $\tilde{\cal Q}^i=\tilde{\cal F}_R/\tilde{\cal F}_R^i$.
Then $\tilde{\cal Q }^i_k\simeq \tilde{\cal G}$    and
  there are exact sequences
    $$
      0\;\longrightarrow\; \tilde{\cal G}\rightarrow \tilde{\cal Q}^{i+1}\;
	      \longrightarrow\; \tilde{\cal Q}^i\;  \longrightarrow\; 0
    $$	
  for all $i$.
It follows from the Local Flatness Criterion [H-L:Lemma~2.1.3] that
  $\tilde{\cal Q}^i$ is a $R/(t^i)$-flat quotient
   of $\tilde{\cal F}_R/((t^i)\tilde{\cal F}_R)$ for all $i$.
Hence the image of the proper morphism
  $$
     \pi_{R}\; :\;
	   \Quot^{Z^{B+\sqrt{-1},L}}_{(C_R\times Y)/R}
            (\tilde{\cal F},  Z^{B+\sqrt{-1}J,L}(\tilde{\cal G})) \;
        \longrightarrow\;  \Spec R
  $$
   contains the closed subscheme $\Spec(R/(t^i))$ for all $i$.
 This implies that
  $\pi_R$ is actually surjective   and, hence,
  there exists a field extension $K \subset  K^{\prime}$ of $K$
   such that
     $\tilde{F}_{K^{\prime}}$ admits a destabilizing quotient
 	 with central charge $Z^{B+\sqrt{c}J,L}(\tilde{\cal G})$.
 This contradicts the semistability assumption on $\tilde{\cal F}_K$
   since
    $\tilde{F}_K$ is $Z$-semistable
	if and only if $\tilde{\cal F}_{K^{\prime}}$ is $Z$-semistable.	
 This proves Proposition~3.5  and hence the completeness of
   the stack $\FM^{1,[0];\scriptsizeZss}_{C_{\cal M}/{\cal M}}(Y;c)$.

 \bigskip

Altogether, this proves the compactness of
   $\FM^{1,[0];\scriptsizeZss}_{C_{\cal M}/{\cal M}}(Y;c)$,
 as claimed in Theorem~3.1.

\newpage
\baselineskip 13pt

\end{document}